\documentclass{amsart}

\usepackage{amsmath,amssymb,latexsym, amscd}
\usepackage{exscale, cite, eps fig, graphics}

\newtheorem{theorem}{Theorem}[section]
\newtheorem{lemma}[theorem]{Lemma}

\newtheorem{corollary}[theorem]{Corollary}

\newtheorem*{main-theorem}{Main Theorem}
\newtheorem{assumption}[theorem]{Assumption}
\newtheorem*{remark*}{Remark}
\numberwithin{equation}{section}

\renewcommand{\geq}{\geqslant}
\renewcommand{\leq}{\leqslant}

\renewcommand{\L}{\mathcal{L}}
\newcommand{\M}{\mathcal{M}}

\renewcommand{\l}{\langle}
\renewcommand{\r}{\rangle}

\newcommand{\e}{\boldsymbol{e}}
\renewcommand{\u}{\boldsymbol{u}}
\renewcommand{\v}{\boldsymbol{v}}
\renewcommand{\P}{\boldsymbol{\phi}}

\begin{document}

\title[Modulational instabillity]{Modulational instability in nonlinear nonlocal equations of regularized long wave type}

\author[Hur]{Vera~Mikyoung~Hur}
\address{Department of Mathematics, University of Illinois at Urbana-Champaign, Urbana, Illinois 61801}
\email{verahur@math.uiuc.edu}

\author[Pandey]{Ashish~Kumar~Pandey}
\email{akpande2@illinois.edu}

\date{\today}

\begin{abstract}
We study the stability and instability of periodic traveling waves in the vicinity of the origin in the spectral plane,
for equations of Benjamin-Bona-Mahony (BBM) and regularized Boussinesq types permitting nonlocal dispersion.
We extend recent results for equations of Korteweg-de Vries type  
and derive modulational instability indices as functions of the wave number of the underlying wave. 
We show that a sufficiently small, periodic traveling wave of the BBM equation 
is spectrally unstable to long wavelength perturbations if the wave number is greater than a critical value
and a sufficiently small, periodic traveling wave of the regularized Boussinesq equation 
is stable to square integrable perturbations.
\end{abstract}

\maketitle

\section{Introduction}\label{sec:intro}
We study the stability and instability of periodic traveling waves for some classes of nonlinear dispersive equations, 
in particular, equations of Benjamin-Bona-Mahony (BBM) type
\begin{equation}\label{E:bbm}
u_t+\mathcal{M}(u+u^2)_x=0
\end{equation} 
and regularized Boussinesq type
\begin{equation}\label{E:bnesq}
u_{tt}-\mathcal{M}^2(u+u^2)_{xx}=0.
\end{equation}
Here, $t\in\mathbb{R}$ is typically proportional to elapsed time 
and $x\in\mathbb{R}$ is the spatial variable in the primary direction of wave propagation;
$u=u(x,t)$ is real valued, representing the wave profile or a velocity.
Throughout we express partial differentiation either by a subscript or using the symbol $\partial$.
Moreover $\mathcal{M}$ is a Fourier multiplier, defined via its symbol as
\[
\widehat{\mathcal{M} f}(k)=m(k) \widehat{f}(k)
\]
and characterizing dispersion in the linear limit. Note that
\begin{equation}\label{def:speeds}
m(k)=\text{the phase speed}\quad\text{and}\quad (km(k))'=\text{the group speed}.
\end{equation}
Throughout the prime means ordinary differentiation.

\begin{assumption}\label{A:m}\rm
We assume that
\begin{itemize}
\item[(M1)] $m$ is real valued and twice continuously differentiable,
\item[(M2)] $m$ is even and, without loss of generality, $m(0)=1$,
\item[(M3)] $C_1|k|^{\alpha}<m(k)<C_2|k|^{\alpha}$ for $|k|\gg1$ for some $\alpha\geq-1$ and $C_1,C_2>0$,
\item[(M4)] $m(k)\neq m(nk)$ for all $k>0$ and $n=2,3,\dots$.
\end{itemize}
\end{assumption}

Assumption (M1) ensures that the spectra of the associated linearized operators 
depend in the $C^1$ manner on the (long wavelength) perturbation parameter;
here we are not interested in achieving a minimal regularity requirement. 
Assumption (M2) is to break that \eqref{E:bbm}, or \eqref{E:bnesq}, is invariant under spatial translations.
Assumption (M3) ensures that periodic traveling waves of \eqref{E:bbm} or \eqref{E:bnesq} are smooth, among others.
Assumption (M4) rules out the resonance between the fundamental mode and a higher harmonic.

The present treatment works {\em mutatis mutandis} for a broad class of nonlinearities. 
Here we assume for simplicity the quadratic power-law nonlinearity.
Incidentally it is characteristic of numerous wave phenomena.

\

In the case of $\mathcal{M}=(1-\partial_x^2)^{-1}$, note that 
\eqref{E:bbm} reduces to the BBM equation
\begin{equation}\label{E:bbm1}
u_t-u_{xxt}+u_x+(u^2)_x=0,
\end{equation}
which was proposed in \cite{BBM}, as an alternative to the Korteweg-de Vries (KdV) equation
\begin{equation}\label{E:kdv1}
u_t+u_x+u_{xxx}+(u^2)_x=0,
\end{equation}
to model long waves of small but finite amplitude in a channel of water.
In the case of $\mathcal{M}^2=(1-\partial_x^2)^{-1}$, moreover, 
\eqref{E:bnesq} reduces to the regularized Boussinesq equation
\begin{equation}\label{E:bnesq1}
u_{tt}=u_{xxtt}+u_{xx}+(u^2)_{xx}.
\end{equation}
It does not explicitly appear in the work of Boussinesq.  
But (280) in \cite{Bnesq1877}, for instance, after several ``higher order terms" drop out,
becomes equivalent to what Whitham derived in \cite[Section~13.11]{Whitham}. 
Under the assumption that $u_t+u_x$ is small (which implies right running waves),
one may, in turn, derive \eqref{E:bnesq1}, or the singular Boussinesq equation
\begin{equation}\label{E:bnesq1'}
u_{tt}=u_{xxxx}+u_{xx}+(u^2)_{xx}.\tag{\ref{E:bnesq1}'}
\end{equation}
Moreover \eqref{E:bnesq1} finds relevance in other physical situations such as nonlinear waves in lattices; 
see \cite{Rosenau}, for instance. 
The phase speed of a plane wave solution with the wave number $k$ of the linear part of \eqref{E:bnesq1} is 
(see \eqref{def:speeds})
\[
\sqrt{\frac{1}{1+k^2}}=1-\frac12k^2+O(k^4)\qquad\text{for}\quad k\ll 1,
\]
and it agrees up to the second order with the phase speed $\sqrt{1-k^2}$ for \eqref{E:bnesq1'} 
when $k$ is small. Hence \eqref{E:bnesq1} and \eqref{E:bnesq1'} are equivalent for long waves. 
But \eqref{E:bnesq1} is preferable over \eqref{E:bnesq1'} for short and intermediately long waves.
As a matter of fact, the initial value problem associated with the linear part of \eqref{E:bnesq1'} is ill-posed, 
because a plane wave solution with $k>1$ grows unboundedly, 
whereas arbitrary initial data lead to short time existence for \eqref{E:bnesq1}. 
Note that \eqref{E:bnesq} factorizes into two sets of \eqref{E:bbm} --- one moving to the left and the other to the right.

Related to \eqref{E:bbm} and \eqref{E:bnesq} are equations of KdV type 
\begin{equation}\label{E:kdv}
u_t=(\mathcal{M}u+u^2)_x.
\end{equation} 
Note that \eqref{E:bbm}, \eqref{E:bnesq} and \eqref{E:kdv} share the dispersion relation in common, 
but their nonlinearities are different. They are {\em nonlocal} 
unless $m$, or $m^{-1}$ in the case of \eqref{E:bbm} and \eqref{E:bnesq}, is a polynomial in $ik$. 
Examples include the Benjamin-Ono equation, for which $m(k)=|k|$ in \eqref{E:kdv},
and the intermediate long wave equation, for which $m(k)=k\coth k$ in \eqref{E:kdv}. 
Another example, which Whitham proposed in \cite{Whitham} to argue for wave breaking in shallow water, 
corresponds to $m(k)=\sqrt{\tanh k/k}$ in \eqref{E:kdv}; see \cite{Hur-breaking}, for instance, for details.

\

By a traveling wave of \eqref{E:bbm}, \eqref{E:bnesq} or \eqref{E:kdv}, we mean a solution 
which progresses at a constant speed without change of form. 
For a broad class of dispersion symbols, periodic traveling waves with small amplitude
may be attained from a perturbative argument, for instance, a Lyapunov-Schmidt reduction; 
see Appendix~\ref{sec:existence} for details. 
We are interested in their stability and instability in the vicinity of the origin in the spectral plane.
Physically, it amounts to long wavelength perturbations or slow modulations of the underlying wave.

Whitham in \cite{Whitham1965, Whitham1967} (see also \cite{Whitham}) developed a formal asymptotic approach 
to study the effects of slow modulations in nonlinear dispersive waves. 
Since then, there has been considerable interest in the mathematical community 
in rigorously justifying predictions from Whitham's modulation theory. 
Recently in \cite{BHV, J2013, HJ2, HJ3} (see also \cite{BHJ}), in particular, 
long wavelength perturbations were carried out analytically
for \eqref{E:kdv} and for a class of Hamiltonian systems permitting nonlocal dispersion,
for which Evans function techniques and other ODE methods may not be applicable.
Specifically, modulational instability indices were derived
either with the help of variational structure (see \cite{BHV}) 
or using asymptotic expansions of the solution (see \cite{J2013, HJ2, HJ3}). 

\begin{theorem}[\cite{HJ2,HJ3}]\label{thm:kdv} 
Under Assumption~\ref{A:m}, a $2\pi/k$-periodic traveling wave of \eqref{E:kdv} with sufficiently small amplitude
is spectrally unstable with respect to long wavelength perturbations if 
\begin{equation} \label{def:ind-kdv}
\text{ind}_{\text{KdV}}(k):=\frac{{i_1(k)} i_2^-(k)i_{\text{KdV}}(k)}{i_3^-(k)}<0,
\end{equation}
where
\begin{align}
i_1(k)=&(km(k))'', \notag \\
i_2^-(k)=&(km(k))'-1, \label{def:i123}\\
i_3^-(k)=&m(k)-m(2k) \notag 
\intertext{and}
i_{\text{KdV}}(k)=&2i_3^-(k)+i_2^-(k).\label{def:i-kdv}
\end{align}
Otherwise, it is stable to square integrable perturbations in the vicinity of the origin in the spectral plane.
\end{theorem}

Here we take matters further and derive modulational instability indices for \eqref{E:bbm} and \eqref{E:bnesq}. 

\begin{theorem}[Modulational instability index for \eqref{E:bbm}]\label{thm:bbm} 
Under Assumption~\ref{A:m}, a sufficiently small, $2\pi/k$-periodic traveling wave of \eqref{E:bbm} 
is spectrally unstable to long wavelength perturbations if 
\begin{equation} \label{def:ind-bbm}
\text{ind}_{\text{BBM}}(k):=\frac{{i_1(k)} i_2^-(k)i_{\text{BBM}}(k)}{i_3^-(k)}<0,
\end{equation}
where $i_1, i_2^-, i_3^-$ are in \eqref{def:i123} and 
\begin{equation}\label{def:i-bbm}
i_{\text{BBM}}(k)=2i_3^-(k)+m(2k)i_2^-(k).
\end{equation}
Otherwise, it is stable to square integrable perturbations in the vicinity of the origin in the spectral plane.
\end{theorem}

\begin{theorem}[Modulational instability index for \eqref{E:bnesq}]\label{thm:bnesq} 
Under Assumption~\ref{A:m}, a sufficiently small, $2\pi/k$-periodic traveling wave of \eqref{E:bnesq} 
is spectrally unstable to long wavelength perturbations if
\begin{equation} \label{def:ind-bnesq}
\text{ind}_{\text{Bnesq}}(k):=\frac{i_1(k)i_2^-(k)i_2^+(k)i_{\text{Bnesq}} (k)}{i_3^-(k) i_3^+(k)}<0,
\end{equation}
where $i_1,i_2^-, i_3^-$ are in \eqref{def:i123}, 
\begin{align}
\begin{split}\label{def:i+}
i_2^+(k)=&(km(k))'+1, \\
i_3^+(k)=&m(k)+m(2k)
\end{split}
\intertext{and}
i_{\text{Bnesq}}(k)=&2i_3^-(k) i_3^+(k)+m^2(2k) i_2^-(k)i_2^+(k).\label{def:i-bnesq}
\end{align}
\end{theorem}

Theorem~\ref{thm:kdv} and Theorem~\ref{thm:bbm} identify 
four resonances which cause change in the sign of the modulational instability index, 
and hence change in modulational stability and instability:  
\begin{itemize}
\item[(R1)] $(k m(k))''=0$ at some $k>0$, 
i.e. the group speed (see \eqref{def:speeds}) attains an extremum at some wave number $k$; 
\item[(R2)] $(k m(k))'=1=m(0)$ at some $k>0$, i.e. the group speed coincides with the phase speed  
of the limiting long wave as $k\to0$, resulting in the resonance between long and short waves;
\item[(R3)] $m(k)=m(2k)$ at some $k>0$, 
i.e. the phase speeds of the fundamental mode and the second harmonic coincide, 
resulting in the ``second harmonic resonance";
\item[(R4)] $i_{\text{KdV}}(k), i_{\text{BBM}}(k)=0$ at some $k>0$.
\end{itemize}
Theorem~\ref{thm:bnesq} identifies the same four resonances 
which cause change in the sign of the modulational instability index, but in bidirectional propagation.
In other words, the phase and group speeds are signed quantities.
Resonances (R1) through (R3) are {\em dispersive}, and \eqref{E:bbm}, \eqref{E:bnesq}, and \eqref{E:kdv} share in common.
Resonance (R4), on the other hand, depends on the {\em nonlinearity} of the equation. 
We shall illustrate this in Section~\ref{sec:fdispersion} 
by comparing \eqref{E:bbm}, \eqref{E:bnesq}, and \eqref{E:kdv} with fractional dispersion.

Thanks to the Galilean invariance\footnote{
Note that \eqref{E:kdv}, in the coordinate frame moving at the speed $c$, remains invariant under
\[
u\mapsto u+v,\qquad c\mapsto c+2v
\]
for any $v>0$.}, the result of Theorem~\ref{thm:kdv} depends merely on the wave amplitude,
whereas the results of Theorem~\ref{thm:bbm} and Theorem~\ref{thm:bnesq} depend on the wave height.
A small amplitude, but not necessarily small height, periodic traveling wave of \eqref{E:bbm} or \eqref{E:bnesq}
may be studied in like manner. But the modulational instability indices become quite complicated. 
Hence we do not include them here.

\

The proofs of Theorem~\ref{thm:bbm} and Theorem~\ref{thm:bnesq} 
follow along the same line as the arguments in \cite{HJ2}, for instance,
inspecting how the spectrum at the origin, in the case of the zero Floquet exponent,
varies with small values of the Floquet exponent and the amplitude parameter.
But the proof of Theorem~\ref{thm:bnesq} necessitates some nontrivial modifications.
Specifically, in the case of the zero Floquet exponent and a small but nonzero amplitude parameter, 
we find four eigenfunctions corresponding to the zero eigenvalue of the associated linearized operator.
In the case of a nonzero Floquet exponent and the zero amplitude, on the other hand,
eigenfunctions for the near-zero eigenvalues do vary with the Floquet exponent at the leading order.
Thus we must concoct basis functions which depend continuously on 
small values of the Floquet exponent and the amplitude parameter.
In the case of \eqref{E:bbm} (and \eqref{E:kdv}), to compare, 
eigenfunctions in the case of the zero Floquet exponent
agree, to the leading order, with eigenfunctions for nonzero Floquet exponents.

Theorem~\ref{thm:bnesq} is merely a sufficient condition for modulational instability. 
In case the modulational instability index is positive, 
the associated linearized operator admits either four stable spectra or four unstable ones,
depending on the nature of the roots of a characteristic polynomial,
whose coefficients are made up of inner products of the basis elements near the origin in the spectral plane 
and involve asymptotic expansions of the associated linearized operator for small Floquet exponents.  
We shall discuss in Appendix~\ref{sec:disc} how to classify the roots of a quartic polynomial,
which will help us to derive supplementary instability indices.
We do not include the detail in Theorem~\ref{thm:bnesq}. Instead in Section~\ref{sec:bnesq1},
we shall illustrate how to use it to determine 
the modulational stability and instability for the regularized Boussinesq equation.

\

We shall discuss in Section~\ref{sec:applications} 
some applications of Theorems~\ref{thm:bbm} and Theorem~\ref{thm:bnesq}. 
In particular, we shall show that a sufficiently-small, $2\pi/k$-periodic traveling wave of the BBM equation
is spectrally unstable to long wavelength perturbations if $k>\sqrt{3}$ 
and spectrally stable to square integrable perturbations if $k<2\sqrt{3/5}$. 
To compare, well-known is that periodic traveling waves of the KdV equation (not necessarily of small amplitudes) are spectrally stable. 
Hence the BBM equation appears to qualitatively reproduce the Benjamin-Feir instability\footnote{
A small amplitude, periodic traveling wave in water goes unstable 
if the wave number of the underlying wave times the undisturbed fluid depth exceeds $1.363\dots$;
see \cite{BF, Whitham1967}, for instance.} of Stokes waves when the KdV equation fails.
But Resonance (R1) following Theorem~\ref{thm:bbm} results in the instability in \eqref{E:bbm1},
whereas Resonances (R1) through (R3) do not occur in the water wave problem, for which $m^2(k)=\tanh k/k$
(see \cite{HJ2, HP2}, for instance).
In other words, the modulational instability mechanism in \eqref{E:bbm1} is different from that in water waves.

The result agrees with that in \cite{J2010}, where the author showed that
periodic traveling waves of the BBM equation with sufficiently small wave numbers
(but not necessarily small amplitudes) are modulationally stable. 
The result complements that in \cite{Haragus08}, 
where the author derived a similar modulational instability index 
and showed the stability of periodic traveling waves of \eqref{E:bbm1} with sufficiently small amplitudes,
but for $c>1$. Here we study when $c<1$.

Moreover, we shall show that all sufficiently small, periodic traveling waves of the regularized Boussinesq equation
are stable to square integrable perturbations in the vicinity of the origin in the spectral plane. 
To the best of the authors' knowledge, these are new findings. 

The treatment in Section~\ref{sec:bnesq} may extend to a broad class of systems of nonlinear dispersive equations. 
In a forthcoming work \cite{HP2}, in particular, the authors will propose 
bi-directional Whitham, or Boussinesq-Whitham, equations for shallow water waves
and demonstrate the instability of the Benjamin-Feir kind.
In contrast, \eqref{E:bnesq}, or \eqref{E:bbm}, 
for which $m^2(k)=\tanh k/k$, describing dispersion of water waves, fails to capture such instability.

\subsubsection*{Notation}
Let $L^p_{2\pi}$ in the range $p\in [1,\infty]$ denote the space of 
$2\pi$-periodic, measurable, real or complex valued functions over $\mathbb{R}$ such that 
\[
\|f\|_{L^p_{2\pi}}=\Big(\frac{1}{2\pi}\int^\pi_{-\pi} |f|^p~dx\Big)^{1/p}<+\infty \quad \text{if}\quad p<\infty
\]
and essentially bounded if $p=\infty$. 
Let $H^1_{2\pi}$ denote the space of $L^2_{2\pi}$-functions such that $f' \in L^2_{2\pi}$
and $H^\infty_{2\pi}=\bigcap_{k=0}^\infty H^k_{2\pi}$. For $f \in L^1_{2\pi}$, we write that 
\[
f(z) \sim \sum_{n\in \mathbb{Z}} \widehat{f}_n e^{inz}, \quad\text{where}\quad
\widehat{f}_n=\frac{1}{2\pi}\int^\pi_{-\pi}f(z)e^{-inz}~dz.
\]
If $f \in L^p_{2\pi}$, $p>1$, moreover, then the Fourier series converges to $f$ pointwise almost everywhere. 
For $0<\alpha<1$, we define the Sobolev space of fractional order via the norm 
\[
\|f\|_{H^{\alpha}_{2\pi}}^2=\widehat{f}_0^2+\sum_{n\in\mathbb{Z}}|n|^{2\alpha}|\widehat{f}_n|^2,
\]
and we define the $L^2_{2\pi}$-inner product as
\begin{equation}\label{def:prod1}
\langle f,g\rangle_{L^2_{2\pi}}=\frac{1}{2\pi}\int^\pi_{-\pi} f(z)\overline{g}(z)~dz.
=\sum_{n\in\mathbb{Z}} \widehat{f}_n\overline{\widehat{g}_n}
\end{equation}
Let $L^p_{2\pi}\times L^p_{2\pi}$ in the range $p\in[1,\infty]$ denote the space of pairs of $L^p_{2\pi}$-functions.
We extend the $L^2_{2\pi}$-inner product as
\begin{equation}\label{def:prod2}
\langle (f_1,f_2),(g_1,g_2) \rangle_{L^2_{2\pi}\times L^2_{2\pi}} = \langle f_1,g_1\rangle_{L^2_{2\pi}}+\langle f_2,g_2\rangle_{L^2_{2\pi}}.
\end{equation}

\section{Equations of Benjamin-Bona-Mahony type}\label{sec:bbm}
We discuss how to follow the arguments in \cite{HJ2} to prove Theorem~\ref{thm:bbm}. 
Details are found in \cite{HJ2} and references therein. Hence we merely hit the main points. 

\

By a  traveling wave of \eqref{E:bbm}, we mean a solution of the form $u(x,t)=u(x-ct)$ 
for some $c>0$, the wave speed, and $u$ satisfying by quadrature that 
\[
\mathcal{M} (u+u^2)-cu=(c-1)^2b
\]
for some $b\in\mathbb{R}$. We seek a $2\pi/k$-periodic traveling wave. That is, $k>0$ is the wave number 
and, abusing notation, $u$ is a $2\pi$-periodic function of $z:=kx$, satisfying that
\begin{equation}\label{E:quad-bbm}
\mathcal{M}_{k}(u+u^2)-cu=(c-1)^2b.
\end{equation}
Here and elsewhere,
\begin{equation}\label{def:Mk}
\mathcal{M}_{k} e^{inz}=m(k n)e^{inz}\quad \text{for}\quad n \in \mathbb{Z}
\end{equation}
and it is extended by linearity and continuity. 
Note from (M2) of Assumption~\ref{A:m} that $\M_k$ maps even functions to even functions. 
Note from (M3) of Assumption~\ref{A:m} that
\[
\M_k:H_{2\pi}^s\rightarrow H_{2\pi}^{s-\alpha}\qquad\text{for all $k>0$}\quad\text{for all $s\geq 0$}
\]
is bounded. Consequently, if $u \in H^{1}_{2\pi}$ solves \eqref{E:quad-bbm} 
for some $c>0$, $k>0$ and $b\in\mathbb{R}$ then $u\in H^\infty_{2\pi}$.
In the case of $\alpha<0$, indeed, it follows from the Sobolev inequality that
\[ 
cu=\M_k(u+u^2)-(c-1)^2b\in H_{2\pi}^{1-\alpha}.
\]
In the case of $\alpha>0$, similarly, (see \cite[Proposition~2.2 and Lemma~5.1]{HJ1}, for instance)
\[
\M_k^{-1}u=\frac{1}{c-\M_k}u^2-(c-1)^2b \in H^1.
\]
The claim then follows from a bootstrapping argument.

\

For an arbitrary $k>0$, a straightforward calculation reveals that
\[
u_0(k,c,b)=b(c-1)+O(b^2)
\]
makes a constant solution of \eqref{E:quad-bbm} for all $c>0$ and $|b|$ sufficiently small.
(The other constant solution is $u=(1-b)(c-1)+O(b^2)$, which we discard for the sake of near-zero solutions.)
We are interested in determining at which value of $c$ there bifurcates 
a family of non-constant $H^1_{2\pi}$-solutions, and hence smooth solutions, of \eqref{E:quad-bbm}.
A necessary condition, it turns out, is that 
the linearized operator of \eqref{E:quad-bbm} about $u_0$ allows a nontrivial kernel. 
This is not in general a sufficient condition. But bifurcation does take place if the kernel is one dimensional. 
Under (M4) of Assumption~\ref{A:m}, a straightforward calculation reveals that 
\[
\ker(\M_k(1+2u_0)-c_0)=\text{span}\{\cos z\}
\]
in the sector of even functions in $H^1_{2\pi}$, provided that  
\begin{equation}\label{def:c0}
c_0(k,b):=m(k)(1+2u_0)= m(k)+2bm(k)(m(k)-1)+O(b^2).
\end{equation}
Therefore
\begin{equation}\label{def:u0}
u_0(k,b):=u_0(k,c_0,b)=b(m(k)-1)+O(b^2).
\end{equation}

For arbitrary $k>0$ and $|b|$ sufficiently small, one may then employ a Lyapunov-Schmidt reduction 
and construct a one-parameter family of non-constant, even and smooth solutions of \eqref{E:quad-bbm} 
near $u=u_0(k,b)$ and $c=c_0(k,b)$. Below we summarize the conclusion.
The proof is in Appendix~\ref{sec:existence}.

\begin{lemma}[Existence]\label{lem:existence} 
Under Assumption~\ref{A:m}, for each $k>0$ and $|b|$ sufficiently small, 
a one-parameter family of $2\pi/k$-periodic traveling waves of \eqref{E:bbm} exists
and, abusing notation, 
\[
u(x,t)=u(a,b)(k(x-c(k,a,b)t))=:u(k,a,b)(z)
\]
for $|a|$ sufficiently small; $u$ and $c$ depend analytically on $k$, $a$, $b$, and 
$u$ is smooth, even and $2\pi$-periodic in $z$, and $c$ is even in $a$. Furthermore,
\begin{align}
u(k,a,b)(z)=&b(m(k)-1)+a \cos z  \label{E:u(k,a,b)} \\
&+\frac{1}{2}a^2\Big(\frac{1}{m(k)-1}+\frac{m(2k)}{m(k)-m(2k)}\cos 2z\Big) +O(a(a^2+b)), \notag \\
c(k,a,b)(z)=&m(k)+2bm(k)(m(k)-1) \label{E:c(k,a,b)}\\
&+a^2 m(k)\Big(\frac{1}{m(k)-1}+\frac{1}{2}\frac{m(2k)}{m(k)-m(2k)}\Big)+O(a(a^2+b)) \notag
\end{align}
as $a,b \to 0$.
\end{lemma}

In the remainder of the section we assume that $b=0$; that means, loosely speaking, the wave height is small.
A small amplitude, but not necessarily small height, periodic traveling wave of \eqref{E:bbm} may be studied
in like manner. But expressions become quite complicated. Hence we do not pursue here.
Let $u=u(k,a,0)$ and $c=c(k,a,0)$ for $k>0$ and $|a|$ sufficiently small, be as in Lemma~\ref{lem:existence}. 
We are interested in its stability and instability. 

\

Linearizing \eqref{E:bbm} about $u$ in the coordinate frame moving at the speed $c$, we arrive at that
\[
v_t+k\partial_z(\mathcal{M}_{k}(1+2u)-c)v=0.
\]
Seeking a solution of the form $v(z,t)=e^{\lambda k t}v(z)$, $\lambda\in\mathbb{C}$ and $v\in L^2(\mathbb{R})$, 
moreover, we arrive at that
\begin{equation}\label{E:eigen-bbm}
\lambda v=\partial_z(-\mathcal{M}_{k}(1+2u)+c)v=:\mathcal{L}(k,a)v.
\end{equation}
We say that $u$ is {\em spectrally unstable} if 
the $L^2(\mathbb{R})$-spectrum of $\L$ intersects the open, right half plane of $\mathbb{C}$
and it is {\em stable} otherwise. Note that $v$ need not have the same period as $u$.
Since the spectrum of $\L$ is symmetric with respect to the reflections about the real and imaginary axes,
$u$ is spectrally unstable if and only if 
the $L^2(\mathbb{R})$-spectrum of $\L$ is {\em not} contained in the imaginary axis. 

It follows from Floquet theory (see \cite{BHJ}, for instance, and references therein) that 
nontrivial solutions of \eqref{E:eigen-bbm} cannot be integrable over $\mathbb{R}$.
Rather they are at best bounded over $\mathbb{R}$.
Furthermore it is well-known that the $L^2(\mathbb{R})$-spectrum of $\mathcal{L}$ is essential.
In the case of the KdV equation, for instance, the (essential) spectrum of the associated linearized operator
may be studied with the help of Evans function techniques and other ODE methods.
Confronted with a {\em nonlocal} operator, unfortunately, they are not viable to use.
Instead, it follows from Floquet theory (see \cite{BHJ}, for instance, and references therein) 
that $\lambda\in\mathbb{C}$ belongs to the $L^2(\mathbb{R})$-spectrum of $\mathcal{L}$ if and only if 
\begin{equation}\label{E:Lxi}
\lambda\phi=e^{-i\xi z}\mathcal{L}(k,a)e^{i\xi z}\phi=:\mathcal{L}_\xi(k,a)\phi
\end{equation}
for some $\xi\in[-1/2,1/2)$ and $\phi\in L^2_{2\pi}$. For each $\xi\in[-1/2,1/2)$, 
the $L^2_{2\pi}$-spectrum of $\mathcal{L}_\xi$ comprises of discrete eigenvalues of finite multiplicities. 
Moreover
\begin{equation*}
\text{spec}_{L^2(\mathbb{R})}(\mathcal{L}(k,a))=
\bigcup_{\xi\in[-1/2,1/2)}\text{spec}_{L^2_{2\pi}}(\mathcal{L}_\xi(k,a)).
\end{equation*}
In other words, the continuous $L^2(\mathbb{R})$-spectrum of $\mathcal{L}$ may be parametrized 
by the family of discrete $L^2_{2\pi}$-spectra of $\mathcal{L}_\xi$'s. Since 
\[
\text{spec}_{L^2_{2\pi}}(\L_\xi)=\overline{\text{spec}_{L^2_{2\pi}}(\L_{-\xi})},
\]
it suffices to take $\xi\in [0,1/2]$.

\subsubsection*{Notation}
In the remainder of the section, $k>0$ is fixed and suppressed to simplify the exposition, unless specified otherwise. Let
\[
\L_{\xi,a}=\L_\xi(k,a).
\]

\medskip

The eigenvalue problem \eqref{E:Lxi} must in general be investigated numerically. 
But in case when $\lambda$ is near the origin and $\xi$ is small, 
we may take a perturbation theory approach in \cite{HJ2}, for instance, and address it analytically.
Specifically, we first study the spectrum of $\mathcal{L}_{0,a}$ at the origin.
We then examine how the spectrum near the origin of $\mathcal{L}_{\xi,a}$
bifurcates from that of $\mathcal{L}_{0,a}$ for $\xi$ small.
Note in passing that $\xi=0$ corresponds to the same period perturbations as the underlying wave 
and $\xi$ small physically amounts to long wavelength perturbations or slow modulations of the underlying wave.

\

In the case of $a=0$, namely the zero solution, a straightforward calculation reveals that 
\begin{equation} \label{E:a=0bbm}
\L_{\xi,0} e^{inz}=i\omega_{n,\xi}e^{inz}
\quad \text{for all $n\in\mathbb{Z}$}\quad\text{for all $\xi\in[0,1/2]$},
\end{equation}
where 
\begin{equation}\label{def:w-bbm}
\omega_{n,\xi}=(\xi+n)(m(k)-m(k(\xi+n))).
\end{equation}
In particular, the zero solution of \eqref{E:bbm} is spectrally stable to square integrable perturbations.
Observe that 
\[
\omega_{1,0}=\omega_{-1,0}=\omega_{0,0}=0,
\]
and $\omega_{n,0}\neq 0$ otherwise. Therefore, 
zero is an $L^2_{2\pi}$-eigenvalue of $\mathcal{L}_{0,0}$ with algebraic and geometric multiplicity three, and
\begin{equation}\label{E:eigen0-bbm}
\cos z, \qquad \sin z\quad\text{and}\quad 1
\end{equation}
form a (real-valued) orthogonal basis of the corresponding eigenspace. 
For $\xi$ small (and $a=0$), furthermore, 
they form an orthogonal basis of the spectral subspace associated with 
eigenvalues $i\omega_{1,\xi}$, $i\omega_{-1,\xi}$, $i\omega_{0,\xi}$ of $\mathcal{L}_{\xi,0}$. 

For $|a|$ small but $\xi=0$, on the other hand, zero is a generalized $L_{2\pi}^2$-eigenvalue of $\mathcal{L}_{0,a}$ 
with algebraic multiplicity three and geometric multiplicity two, and 
\begin{align}
\phi_1(z) &=:\frac{1}{2m(k)(m(k)-1)}((\partial_b c)(\partial_a u)-(\partial_a c)(\partial_b u))(k,a,0)(z) \notag\\
&=\cos z-\frac{1}{2}a\frac{m(2k)}{m(k)-m(2k)}+a\frac{m(2k)}{m(k)-m(2k)}\cos 2z+O(a^2)\label{def:p1}\\
\phi_2(z) &=:-\frac{1}{a}\partial_z u(k,a,0)(z)= \sin z+a\frac{m(2k)}{m(k)-m(2k)}\sin 2z+O(a^2)\label{def:p2}\\
\phi_3(z) &=:\frac{1}{m(k)-1}\partial_b u(k,a,0)(z)= 1+O(a^2)\label{def:p3}
\end{align}
form a basis of the corresponding generalized eigenspace. 
Indeed, differentiating \eqref{E:quad-bbm} with respect to $z$, $a$, $b$, we find that
\[
\mathcal{L}_{0,a}(\partial_z u)=0,\qquad
\mathcal{L}_{0,a}(\partial_a u)=(\partial_a c)(\partial_z u),\qquad
\mathcal{L}_{0,a}(\partial_b u)=(\partial_b c)(\partial_z u),
\]
respectively, and \eqref{def:p1}-\eqref{def:p3} follows at once; see \cite[Lemma~3.1]{HJ2} for details. 
In the case of $a=0$, note that \eqref{def:p1}-\eqref{def:p3} reduce to \eqref{E:eigen0-bbm}.

\

To recapitulate, in the case of $\xi$ small and $a=0$, $\L_{\xi,0}$ possesses three purely imaginary eigenvalues
near the origin and functions in \eqref{E:eigen0-bbm} form an orthogonal basis of the associated spectral subspace.
In the case of $\xi=0$ and $a$ small, moreover, $\L_{0,a}$ possesses three eigenvalues at the origin
and functions in \eqref{def:p1}-\eqref{def:p3} form a basis of the associated eigenspace. 
In order to study how three eigenvalues at the origin vary with $\xi$ and $|a|$ small,
we proceed as in \cite{HJ2} and compute $3\times3$ matrices
\begin{equation}\label{def:BI}
\mathbf{B}_{\xi,a}=\left( \frac{\l \L_{\xi,a}\phi_j, \phi_k\r}{\l \phi_j, \phi_j\r}\right)_{j,k=1,2,3}
\quad\text{and}\quad
\mathbf{I}_{a}=\left( \frac{\l \phi_j, \phi_k\r}{\l \phi_j, \phi_j\r}\right)_{j,k=1,2,3},
\end{equation}
where $\phi_j$'s, $j=1,2,3$, are in \eqref{def:p1}-\eqref{def:p3} 
and $\langle\,,\rangle=\langle\,,\rangle_{L^2_{2\pi}}$ is in \eqref{def:prod1}. 
Note that $\mathbf{B}_{\xi,a}$ and $\mathbf{I}_{a}$, respectively, represent actions of 
$\mathcal{L}_{\xi,a}$ and the identity on the spectral subspace associated with three eigenvalues at the origin.
For $\xi$ and $|a|$ sufficiently small, eigenvalues of $\L_{\xi,a}$ agree in location and multiplicity
with the roots of the characteristic equation $\det(\mathbf{B}_{\xi,a}-\lambda\mathbf{I}_{a})=0$;
see \cite[Section~4.3.5]{K}, for instance, for details.

Using \eqref{E:eigen-bbm}, \eqref{E:Lxi} and \eqref{E:u(k,a,b)}, \eqref{E:c(k,a,b)}, 
we make a Baker-Campbell-Hausdorff expansion to write that 
\begin{align}
\mathcal{L}_{\xi,a}
=&\mathcal{L}_{0,0}+i\xi [\mathcal{L}_{0,0},z]-\frac{\xi^2}{2}[[\mathcal{L}_{0,0},z],z]\notag\\
&-2a\M_k \partial_z (\cos z)-2i\xi a[\partial_z\M_k,z]\cos z+O(\xi^3+\xi^2 a+a^2)\notag\\
=:&L-2a\M_k \partial_z (\cos z)-2i\xi a M_1\cos z+O(\xi^3+\xi^2 a+a^2) \label{def:L-bbm}
\end{align}
as $\xi, a\to 0$. 
Note that $M_1=[\mathcal{L}_{0,0},z]$ and $[[\mathcal{L}_{0,0},z],z]$ are well defined in $L^2_{2\pi}$ 
even though $z$ is not. Note moreover that $L=\mathcal{L}_{\xi,0}$ up to the second order for $\xi\ll 1$
and $M_1$ is the $O(\xi)$ term in the asymptotic expansion of $\mathcal{L}_{\xi,0}$ for $\xi\ll1$.

We use \eqref{E:a=0bbm} and \eqref{def:w-bbm}, or its Taylor expansion (see (M1) of Assumption~\ref{A:m}),
to compute that
\begin{align*}
\L_{\xi,0} e^{\pm inz}=\pm in(m(k)-m(kn))e^{\pm inz}&+i\xi(m(k)-m(kn)-km'(kn))e^{\pm inz} \\
&\mp \frac12\xi^2(2km'(kn)+k^2m''(kn))e^{\pm inz}+O(\xi^3)
\end{align*}
as $\xi\to 0$.
Therefore we infer that
\begin{equation*}
L 1=i\xi(m(k)-1)\qquad\text{and}\quad M_11= m(k)-1.
\end{equation*}
Similarly,
\begin{align*}
L\left\{ \begin{matrix} \cos z\\ \sin z\end{matrix}\right\}=&
-i\xi km'(k)\left\{ \begin{matrix} \cos z\\ \sin z\end{matrix}\right\} 
\pm \frac12\xi^2(2km'(k)+k^2m''(k))\left\{ \begin{matrix} \sin z\\ \cos z\end{matrix}\right\},\\
M_1\left\{ \begin{matrix} \cos z\\ \sin z\end{matrix}\right\}
=&-km'(k)\left\{ \begin{matrix} \cos z\\ \sin z\end{matrix}\right\} \notag
\end{align*} 
and
\begin{align*}
L\left\{ \begin{matrix} \cos 2z\\ \sin 2z\end{matrix}\right\}=&
\mp 2(m(k)-m(2k))\left\{ \begin{matrix} \sin 2z\\ \cos 2z\end{matrix}\right\}
+i\xi(m(k)-m(2k)-km'(2k))\left\{ \begin{matrix} \cos 2z\\ \sin 2z\end{matrix}\right\} \\
&\hspace*{130pt}\pm \frac12\xi^2(2km'(2k)+k^2m''(2k))\left\{ \begin{matrix} \sin 2z\\ \cos 2z\end{matrix}\right\},\\
M_1\left\{ \begin{matrix} \cos 2z\\ \sin 2z\end{matrix}\right\}=&(m(k)-m(2k)-km'(2k))\left\{ \begin{matrix} \cos 2z\\ \sin 2z\end{matrix}\right\}. \notag
\end{align*} 

Substituting \eqref{def:p1}-\eqref{def:p3} into \eqref{def:L-bbm},
and using the above and \eqref{E:a=0bbm}, we make a lengthy but straightforward calculation to find that
\begin{align*}
\mathcal{L}_{\xi,a} \phi_1 =&-i\xi k m'(k)\cos z \\ &-i\xi a\Big( 1+\frac{m(2k)(m(k)-1)}{2(m(k)-m(2k))}\Big)\\
&-i\xi a\Big( m(2k)+2km'(2k)-\frac{m(2k)(m(k)-m(2k)-2km'(2k))}{m(k)-m(2k)}\Big)\cos 2z\notag \\
&+\frac12\xi^2(2km'(k)+k^2m''(k))\sin z+O(\xi^3+a^2)
\intertext{and}
\mathcal{L}_{\xi,a} \phi_2 =&-i\xi k m'(k)\sin z \\
&-i\xi a\Big( m(2k)+2km'(2k)-\frac{m(2k)(m(k)-m(2k)-2km'(2k))}{m(k)-m(2k)}\Big)\sin 2z\notag\\
&-\frac12\xi^2(2km'(k)+k^2m''(k))\cos z+O(\xi^3+a^2), \notag\\
\mathcal{L}_{\xi,a} \phi_3=&2a m(k) \sin z+i\xi (m(k)-1)-2i\xi a(m(k)+km'(k))\cos z+O(\xi^3+a^2) 
\end{align*}
as $\xi, a\to 0$. Recall \eqref{def:prod1}. 
Using the above and \eqref{def:p1}-\eqref{def:p3}, we make another lengthy but straightforward calculation to find that
\begin{align*}
\langle \mathcal{L}_{\xi,a}\phi_1,\phi_1\rangle &=\langle  \mathcal{L}_{\xi,a}\phi_2,\phi_2 \rangle
= -\frac12i\xi k m'(k)+O(\xi^3+a^2),\\
\langle \mathcal{L}_{\xi,a}\phi_1,\phi_2\rangle &=-\langle \mathcal{L}_{\xi,a}\phi_2,\phi_1\rangle
= \frac14\xi^2(2km'(k)+k^2m''(k))+O(\xi^3+a^2),\\
\langle \mathcal{L}_{\xi,a}\phi_1,\phi_3\rangle &= 
-i\xi a\Big(1+\frac12\frac{m(2k)(m(k)-1)}{m(k)-m(2k)}\Big)+O(\xi^3+a^2)
\intertext{and}
\langle \mathcal{L}_{\xi,a}\phi_2,\phi_3\rangle &=0+O(\xi^3+a^2),\\
\langle \mathcal{L}_{\xi,a}\phi_3,\phi_1\rangle &=-i\xi a\Big( m(k)+km'(k)+\frac12\frac{m(2k)(m(k)-1)}{m(k)-m(2k)}\Big)+O(\xi^3+a^2),\\
\langle \mathcal{L}_{\xi,a}\phi_3,\phi_2\rangle &=a m(k)+O(\xi^3+a^2),\\
\langle \mathcal{L}_{\xi,a}\phi_3,\phi_3\rangle &=i\xi( m(k)-1)+O(\xi^3+a^2)
\end{align*}
as $\xi,a\rightarrow 0$. Moreover we use \eqref{def:p1}-\eqref{def:p3} to compute that
\begin{align*}
\langle \phi_1,\phi_1\rangle &= \langle \phi_2,\phi_2 \rangle=\frac{1}{2}+O(\xi^3+a^2),\\
\langle \phi_1,\phi_2\rangle &=0+O(\xi^3+a^2),\\
\langle \phi_1,\phi_3\rangle &=-a\frac12\frac{m(2k)}{m(k)-m(2k)}+O(\xi^3+a^2), \\
\langle \phi_2,\phi_2\rangle &=0+O(\xi^3+a^2),  \\
\langle \phi_3,\phi_3\rangle &=1+O(\xi^3+a^2)
\end{align*}
as $\xi,a\rightarrow 0$. To summarize, \eqref{def:BI} becomes
\begin{align}\label{E:B-bbm}
\mathbf{B}_{\xi,a} &=a m(k)\begin{pmatrix} 0&0&0\\0&0&0\\0&1&0\end{pmatrix}\\
&+i\xi \begin{pmatrix} -km'(k)&0&0\\ 0&-km'(k)&0\\ 0&0&m(k)-1\end{pmatrix}\label{E:B-bbm} \notag\\
&-i\xi a\begin{pmatrix} 0&0&{\displaystyle 2+\frac{m(2k)(m(k)-1)}{m(k)-m(2k)}}\\ 0&0&0\\ 
{\displaystyle m(k)+km'(k)+\frac12\frac{m(2k)(m(k)-1)}{m(k)-m(2k)}}&0&0\end{pmatrix}\notag\\
&+\xi^2(km'(k)+\tfrac12k^2m''(k))\begin{pmatrix}0&1&0\\ -1&0&0\\ 0&0&0 \end{pmatrix}+O(\xi^3+a^2)\notag
\end{align}
and
\begin{equation}
\mathbf{I}_{a}=\mathbf{I}-a\frac{m(2k)}{2(m(k)-m(2k))}\begin{pmatrix} 0&0&2\\0&0&0\\1&0&0\end{pmatrix}+O(a^2)\label{E:I-bbm}
\end{equation}
as $\xi,a\to 0$. Here $\mathbf{I}$ denotes the $3\times3$ identity matrix.

\

We turn the attention to the roots of the characteristic polynomial
\begin{align*}
\det(\mathbf{B}_{\xi,a}-\lambda\mathbf{I}_{a})=D_3(\xi,a)\lambda^3+iD_2(\xi,a)\lambda^2+D_1(\xi,a)\lambda+iD_0(\xi,a)
\end{align*}
for $\xi$ and $|a|$ sufficiently small, 
where $\mathbf{B}_{\xi,a}$ and $\mathbf{I}_{a}$ are in \eqref{E:B-bbm} and \eqref{E:I-bbm}.
Details are found in \cite[Section~3.3]{HJ2}. Hence we merely hit the main points.

Observe that $D_j=\xi^{3-j}d_j$, $j=0,1,2,3$, for some real $d_j$'s.
We may therefore write that
\[
\det(\mathbf{B}_{\xi,a}-(-i\xi)\lambda\mathbf{I}_{a})= 
i\xi^3 (d_3(\xi,a)\lambda^3-d_2(\xi,a)\lambda^2-d_1(\xi,a)\lambda+d_0(\xi,a)).
\]
The underlying, periodic traveling wave of \eqref{E:bbm} is modulationally unstable 
if $\det(\mathbf{B}_{\xi,a}-(-i\xi)\lambda\mathbf{I}_{a})$ admits a pair of complex roots, or equivalently,
the discriminant of the cubic polynomial
\[
\text{disc}_{\text{BBM}}(\xi,a):=18d_3d_2d_1d_0+d_2^2d_1^2+4d_2^3d_0+4d_3d_1^3-27d_3^2d_0^2<0
\]
for $\xi$ and $|a|$ sufficiently small, while it is modulationally stable if $\text{disc}_{\text{BBM}}(\xi,a)>0$. 
Observe that $\text{disc}_{\text{BBM}}(\xi,a)$ is even in $\xi$ and $a$, whereby we may write that 
\[
\text{disc}_{\text{BBM}}(\xi,a):=\text{disc}_{\text{BBM}}(k;\xi,0)+\text{ind}_{\text{BBM}}(k)a^2+O(a^2(a^2+\xi^2))
\]
as $\xi,a\to 0$. 
It is readily seen from \eqref{E:B-bbm} and \eqref{E:I-bbm} that $\text{disc}_{\text{BBM}}(k;\xi,0)>0$ for all $k>0$.
Specifically, a Mathematica calculation reveals that
\[
\text{disc}_{\text{BBM}}(k;\xi,0)=\frac{1}{16}\xi^2( ki_1(k)(ki_1(k)\xi-4i_2^-(k))(ki_1(k)\xi+4i_2^-(k)))^2.
\]
Therefore the sign of $\text{ind}_{\text{BBM}}(k)$ determines modulational stability and instability. 
As a matter of fact, if $\text{ind}_{\text{BBM}}(k)<0$ then $\text{disc}_{\text{BBM}}(k;\xi,a)<0$ 
for $\xi$ sufficiently small, depending on $a$ sufficiently small but fixed, implying modulational instability,
whereas if $\text{ind}_{\text{BBM}}(k)>0$ then $\text{disc}_{\text{BBM}}(k;\xi,a)>0$ 
for all $k$ and $\xi, |a|$ sufficiently small, implying modulational stability.
Recalling \eqref{E:B-bbm} and \eqref{E:I-bbm}, 
a Mathematica calculation then reveals that 
the sign of $\text{ind}_{\text{BBM}}(k)$ agrees with that of \eqref{def:ind-bbm}. 
This completes the proof of Theorem~\ref{thm:bbm}.

\section{Equations of regularized Boussinesq type}\label{sec:bnesq}

We discuss how to extend the arguments in \cite{HJ2} and the previous section to prove Theorem~\ref{thm:bnesq}.
It is convenient to write \eqref{E:bnesq}, equivalently, in the Hamiltonian form\footnote{
The present development does not rely on the Hamiltonian structure. 
But \eqref{E:main} puts the associated spectral problem in the traditional form,
where the spectral parameter dapperly linearly.}
\begin{equation}\label{E:main}
\begin{cases}
u_t=\mathcal{M}^2 q_x, \\
q_t=(u+u^2)_x.\\
\end{cases}
\end{equation}
Throughout the section, $\u=(u,q)$. 

\subsection{Remark on periodic traveling waves} 

We seek a $2\pi/k$-periodic traveling wave of \eqref{E:bnesq}, and hence \eqref{E:main}. 
That is, $u(x,t)=u(k(x-ct))$, where $c>0$ is the wave speed, $k>0$ is the wave number,
and, abusing notation, $u$ is a $2\pi$-periodic function of $z:=kx$, satisfying by quadrature that
\begin{equation}\label{E:quad-bnesq}
\mathcal{M}_k^2(u+u^2)-c^2u=(c^2-1)^2b
\end{equation}
for some $b\in\mathbb{R}$; $\mathcal{M}_k$ is in \eqref{def:Mk}.
Equivalently, $(u,q)$ is a $2\pi$-periodic vector-valued function of $z=kx$, satisfying that 
\begin{equation}\label{E:quad-main}
\begin{cases}
{\displaystyle cu+\mathcal{M}_k^2q+\frac{(c^2-1)^2}{c}b_1=0,}\\
cq+u+u^2+(c^2-1)^2b_2=0
\end{cases}
\end{equation}
for some $b_1,b_2\in \mathbb{R}$. Note that $b=b_1-b_2$. 

Observe that \eqref{E:quad-bnesq} is identical to \eqref{E:quad-bbm} 
after replacing $\M_k$ by $\M_k^2$ and $c$ by $c^2$.
The existence and regularity results in the previous section therefore hold for \eqref{E:quad-bnesq}, 
and hence \eqref{E:quad-main}. Below we summarize the conclusion.

\begin{lemma}[Existence]\label{lem:exist-bnesq} 
Under Assumption~\ref{A:m}, for arbitrary $k>0$ and $|b_1|,|b_2|$ sufficiently small, 
a one-parameter family of $2\pi/k$-periodic traveling waves of \eqref{E:bnesq}, and hence \eqref{E:main}, 
exists and, abusing notation,
\[ 
\u(x,t)=\u(a,b_1,b_2)(k(x-c(k,a,b_1,b_2)t))=:\u(k,a,b_1,b_2)(z) 
\]
for $|a|$ sufficiently small; $\u=(u,q)$ and $c$ depend analytically on $k$, $a$, $b_1$, $b_2$, 
and $u$, $q$ are smooth, even and $2\pi$-periodic in $z$, and $c$ is even in $a$. Furthermore,
\begin{equation}\label{def:u}
\left\{\begin{split}
u(k,a,b_1,b_2)(z)=&u_0(k,b_1,b_2)+a m(k)\cos z +a(b_1-b_2)m(k)(m^2(k)-1)\cos z \\
&+a^2 (U_0+U_2 \cos 2z)+O(a(a^2+(b_1+b_2)^2)), \\
q(k,a,b_1,b_2)(z)=&q_0(k,b_1,b_2)-a \cos z-2a(b_1-b_2)(m^2(k)-1)\cos z \\
&-a^2\Big(m(k)U_0+\frac{m(k)}{m^2(2k)}U_2 \cos 2z\Big)+O(a(a^2+(b_1+b_2)^2)) 
\end{split}\right.
\end{equation}
and
\begin{equation}
c(k,a,b_1,b_2)=c_0(k,b_1,b_2)+a^2m(k)\Big(U_0+\frac12U_2\Big)+O(a(a^2+(b_1+b_2)^2))\label{def:c}
\end{equation}
as $a,b_1,b_2 \to 0$, where
\begin{equation}\label{E:bnesq-u0}
\left\{\begin{split}
u_0(k,b_1,b_2)&=(b_1-b_2)(m^2(k)-1)+O((b_1+b_2)^2),\\
q_0(k,b_1,b_2)&=\Big(-b_1\frac{1}{m(k)}+b_2m(k)\Big)(m^2(k)-1)+O((b_1+b_2)^2), 
\end{split}\right.
\end{equation}
\begin{equation}
c_0(k,b_1,b_2)= m(k)+(b_1-b_2)m(k)(m^2(k)-1)+O((b_1+b_2)^2) \label{E:bnesq-c0},
\end{equation}
and
\begin{equation}\label{def:U02}
U_0=\frac12\frac{m^2(k)}{m^2(k)-1} \quad \text{and} \quad 
U_2=\frac12\frac{m^2(k)m^2(2k)}{m^2(k)-m^2(2k)}.
\end{equation}
\end{lemma}

It remains to show \eqref{def:u}-\eqref{def:c} and \eqref{E:bnesq-u0}-\eqref{E:bnesq-c0}.
For an arbitrary $k>0$, a straightforward calculation reveals that
\[
\begin{cases}
u_0(k,c,b_1,b_2)=(b_1-b_2)(c^2-1)+O((b_1+b_2)^2),\\
q_0(k,c,b_1,b_2)={\displaystyle -b_1\frac{c^2-1}{c}+b_2c(c^2-1)+O((b_1+b_2)^2)}
\end{cases}
\]
form a constant solution of \eqref{E:quad-main} for all $c>0$ and $|b_1|,|b_2|$ sufficiently small. 
Thanks to (M4) of Assumption~\ref{A:m}, it then follows from bifurcation theory that 
a family of non-constant, even and $H^1_{2\pi}\times H^1_{2\pi}$-solutions, 
and hence smooth solutions, of \eqref{E:quad-main} exists, provided that
\[
\begin{pmatrix} c & \M_k^2\\ 1+2u_0  & c \end{pmatrix}
\begin{pmatrix} u_1\\q_1\end{pmatrix}\cos z=\mathbf{0}
\]
for some nontrivial $(u_1,q_1)$. Therefore 
\[
c_0^2=m^2(k)(1+2u_0).
\] 
One may therefore deduce \eqref{E:bnesq-u0}-\eqref{E:bnesq-c0}. Furthermore
\begin{equation}\label{E:uq1}
\begin{pmatrix}u_1\\q_1\end{pmatrix}=\begin{pmatrix}c_0\\-1-2u_0\end{pmatrix}
=\begin{pmatrix}m(k)+(b_1-b_2)m(k)(m^2(k)-1)\\ -1-2(b_1-b_2)(m^2(k)-1)\end{pmatrix}
\end{equation}
up to multiplication by a constant. 

Let $k>0$ be fixed and suppressed, to simplify the exposition. We assume that $b_1=b_2=0$. 
As a matter of fact, it suffices to find $u$ and $q$ up to the linear order in $b_1$ and $b_2$.
Since $u$, $q$ and $c$ depend analytically on $a$ for $|a|$ sufficiently small 
and since $c$ is even in $a$, we write that
\[
\begin{cases}
u(k,a,b_1,b_2)(z)=u_0 (k,b_1,b_2)+u_1(k,b_1,b_2)\cos z+a^2u_2(z)+a^3u_3(z)+\cdots, \\
q(k,a,b_1,b_2)(z)=q_0(k,b_1,b_2)+q_1(k,b_1,b_2)\cos z+a^2q_2(z)+a^3q_3(z)+\cdots
\end{cases}
\]
and
\[
c(k,a,b_1,b_2)=c_0(k,b_1,b_2)+a^2c_2+\cdots
\]
as $a \to 0$, where $u_1$ and $q_1$ are in \eqref{E:uq1},
$u_2$, $u_3$, $q_2$, $q_3,\dots$ are even and $2\pi$-periodic in $z$. 
Substituting these into \eqref{E:quad-main}, at the order of $a^2$, we gather that
\[
\begin{cases}
m(k)u_2+\mathcal{M}_k^2q_2=0,\\
m(k)q_2+u_2+m^2(k)\cos^2 z=0.\\
\end{cases}
\]
A straightforward calculation then reveals that $u_2(z)=U_0+U_2\cos 2z$ and $q_2(z)=Q_0+Q_2\cos 2z$, 
where $U_0$ and $U_2$ are in \eqref{def:U02}, and 
\[ 
\quad Q_0=-m(k)U_0\quad\text{and}\quad Q_2=-\frac{m(k)}{m^2(2k)}U_2.
\]
Continuing, at the order of $a^3$, 
\[
\begin{cases}
m(k)u_3+m(k)c_2\cos z+\mathcal{M}_k^2q_3=0,\\
m(k)q_3-c_2\cos z+u_3+2m(k)u_2\cos z=0,
\end{cases}
\]
whence $c_2=m(k)\Big(U_0+\frac12U_2\Big)$. This completes the proof.

\subsection{Modulational instability}\label{sec:MI-bnesq}

In the remainder of the section we assume that $b_1=b_2=0$; 
that means, loosely speaking, the wave height is small. 
A small amplitude, but not necessarily small height, periodic traveling wave of \eqref{E:bnesq}, 
and hence \eqref{E:main}, may be studied in like manner. But expressions become quite complicated. 
Hence we do not pursue here.

Let $\u=\u(k,a,0,0)$ and $c=c(k,a,0,0)$, for $k>0$ and $|a|$ sufficiently small, 
form a sufficiently small, $2\pi/k$-periodic traveling wave of \eqref{E:main}, 
whose existence follows from Lemma~\ref{lem:exist-bnesq}. 

\

Linearizing \eqref{E:main} about $\u$ in the coordinate frame moving at the speed $c$, we arrive at that
\[
\partial_t \v = k\partial_z \begin{pmatrix}c &\mathcal{M}_k^2\\ 1+2u & c  \end{pmatrix}\v.
\]
Seeking a solution of the form $\v(z,t)=e^{\lambda k t}\v(z)$, 
where $\lambda \in \mathbb{C}$ and $\v\in L^2_{2\pi} \times L^2_{2 \pi}$, moreover, we arrive at that
\begin{equation}\label{E:bnesq-eigen}
\lambda\v = \partial_z \begin{pmatrix}c &\mathcal{M}_k^2\\ 1+2u & c  \end{pmatrix}\v=: \mathcal{L}(k,a)\v.
\end{equation}
We say that $\u$ is spectrally unstable if the $L^2(\mathbb{R})\times L^2(\mathbb{R})$-spectrum of $\L$ 
intersects the open, right half plane of $\mathbb{C}$ and it is stable otherwise. 
Note that $\v$ need not have the same period as $\u$. 
Since \eqref{E:bnesq-eigen} remains invariant under 
\[
\v\mapsto \bar{\v} \quad\text{and}\quad\lambda\mapsto \bar{\lambda}
\] 
and under
\[
z\mapsto -z\quad\text{and}\quad \lambda\mapsto -\lambda,
\]
the spectrum of $\L$ is symmetric with respect to the reflections about the real and imaginary axes. 
Therefore $\u$ is spectrally unstable if and only if 
the $L^2(\mathbb{R})\times L^2(\mathbb{R})$-spectrum of $\L$ is not contained in the imaginary axis.

We repeat the argument in the previous section (see also \cite{HJ2, BHJ} and references therein) to learn that
$\lambda\in\mathbb{C}$ belongs to the $L^2(\mathbb{R})\times L^2(\mathbb{R})$-spectrum of $\mathcal{L}$
if and only if 
\begin{equation}\label{E:bnesq-bloch}
\lambda\P=e^{-i\xi z}\mathcal{L}(k,a)e^{i\xi z}\P=:\mathcal{L}_\xi(k,a)\P
\end{equation}
for some $\xi\in[-1/2,1/2)$ and $\P \in L^2_{2\pi}\times L^2_{2\pi}$. 
For each $\xi\in[-1/2,1/2)$, the $L^2_{2\pi}\times L^2_{2\pi}$-spectrum of $\mathcal{L}_\xi$ 
comprises entirely of discrete eigenvalues of finite multiplicities. Moreover
\[
\text{spec}_{L^2(\mathbb{R})\times L^2(\mathbb{R})}(\mathcal{L}(k,a,b))=
\bigcup_{\xi\in[-1/2,1/2)}\text{spec}_{L^2_{2\pi}\times L^2_{2\pi}}(\mathcal{L}_\xi(k,a,b)).
\]
Since 
\[
\text{spec}_{L^2_{2\pi}\times L^2_{2\pi}}(\L_\xi)=\overline{\text{spec}_{L^2_{2\pi}\times L^2_{2\pi}}(\L_{-\xi})},
\]
it suffices to take $\xi\in [0,1/2]$.

\subsubsection*{Notation} 
In what follows, $k>0$ is fixed and suppressed to simplify the exposition, unless specified otherwise. Let 
\[
\L_{\xi,a}=\L_\xi(k,a).
\]

\subsection{Spectra of $\L_{\xi,a}$'s}\label{sec:bnesq-spec}
We proceed as in \cite{HJ2}, or the previous section, 
to study the $L^2_{2\pi}\times L^2_{2\pi}$-spectra of $\mathcal{L}_{\xi,a}$ near the origin for $\xi, |a|$ sufficiently small.
Recall that $\xi=0$ corresponds to the same period perturbations as the underlying wave
and $\xi$ small physically amounts to long wavelength perturbations or slow modulations of the underlying wave.

\

In the case of $a=0$, namely the zero solution, a straightforward calculation reveals that 
\begin{equation}\label{E:a=0bnesq}
\mathcal{L}_{\xi,0}\e_{n,\xi}^{\pm}=i \omega_{n,\xi}^{\pm}\e_{n,\xi}^{\pm}
\end{equation}
where 
\begin{equation}\label{def:w&e}
\omega_{n,\xi}^{\pm}=(\xi+n)(m(k)\pm m(k(\xi+n)))\quad\text{and}\quad
\e_{n,\xi}^{\pm}(z)=\begin{pmatrix} \pm m(k(\xi+n))\\ 1\end{pmatrix}e^{inz}.
\end{equation}
In particular, the zero solutions of \eqref{E:bnesq}, and hence \eqref{E:main}, is spectrally stable to square integrable perturbations.
Observe that
\[ 
\omega_{0,0}^{+}=\omega_{0,0}^{-}=\omega_{1,0}^{-}=\omega_{-1,0}^{-}=0,
\]
and $\omega_{n,0}^\pm\neq 0$ otherwise. 
As a matter of fact, zero is an $L^2_{2\pi}\times L^2_{2\pi}$-eigenvalue of $\mathcal{L}_{0,0}$ with algebraic and geometric multiplicity four, and 
\begin{equation}\label{E:eigen00}
\begin{split}
&\frac{\e_{1,0}^{-}+\e_{-1,0}^{-}}{2}(z)=\begin{pmatrix} -m(k)\\ 1\end{pmatrix}\cos z, \\
&\frac{\e_{1,0}^{-}-\e_{-1,0}^{-}}{2i}(z)=\begin{pmatrix} -m(k)\\ 1\end{pmatrix}\sin z,\\
&\frac{\e_{0,0}^{+}+\e_{0,0}^{-}}{2}(z)=\begin{pmatrix} 0\\ 1\end{pmatrix},\\
&\frac{\e_{0,0}^{+}-\e_{0,0}^{-}}{2}(z)=\begin{pmatrix} 1\\ 0\end{pmatrix}
\end{split}
\end{equation}
form a (real-valued) orthogonal basis of the corresponding eigenspace. 
For $\xi$ small (and $a=0$), furthermore, 
\begin{equation}\label{E:eigen-xi}
\begin{split}
\P_{1,\xi}(z):=&\frac{\sqrt{1+m^2(k)}}{2}\left(\frac{\e_{1,\xi}^-}{\langle\e_{1,\xi}^-, \e_{1,\xi}^-\rangle^{1/2}}
+\frac{\e_{-1,\xi}^-}{\langle\e_{-1,\xi}^-, \e_{-1,\xi}^-\rangle^{1/2}}\right)(z) \\
=&\begin{pmatrix}-m(k)\\ 1 \end{pmatrix}\cos z 
-i\xi \frac{km'(k)}{1+m^2(k)}\begin{pmatrix} 1\\ m(k)\end{pmatrix}\sin z+\xi^2 \boldsymbol{\alpha}\cos z+O(\xi^3), 
\hspace*{-10pt}\\
\P_{2,\xi}(z):=&\frac{\sqrt{1+m^2(k)}}{2i}\left(\frac{\e_{1,\xi}^-}{\langle\e_{1,\xi}^-, \e_{1,\xi}^-\rangle^{1/2}}
-\frac{\e_{-1,\xi}^-}{\langle\e_{-1,\xi}^-, \e_{-1,\xi}^-\rangle^{1/2}}\right)\\
=&\begin{pmatrix}-m(k)\\ 1 \end{pmatrix}\sin z
+i\xi\frac{ km'(k)}{1+m^2(k)}\begin{pmatrix} 1\\ m(k)\end{pmatrix}\cos z+\xi^2 \boldsymbol{\alpha}\sin z+O(\xi^3),
\hspace*{-10pt}\\
\P_{3,\xi}(z):=&\frac{1}{\sqrt{2}}\left(\frac{\e_{0,\xi}^+}{\langle\e_{0,\xi}^+,\e_{0,\xi}^+\rangle^{1/2}}
+\frac{\e_{0,\xi}^-}{\langle\e_{0,\xi}^-,\e_{0,\xi}^-\rangle^{1/2}}\right)(z) \\
=&\begin{pmatrix}0\\ 1 \end{pmatrix}+\xi^2\frac{k^2}{12}\begin{pmatrix} 0\\1\end{pmatrix}+O(\xi^3), \\
\P_{4,\xi}(z):=&\frac{1}{\sqrt{2}}\left(\frac{\e_{0,\xi}^+}{\langle\e_{0,\xi}^+,\e_{0,\xi}^+\rangle^{1/2}}
-\frac{\e_{0,\xi}^-}{\langle\e_{0,\xi}^-,\e_{0,\xi}^-\rangle^{1/2}}\right)(z) \\
=&\begin{pmatrix}1\\ 0 \end{pmatrix}-\xi^2\frac{k^2}{12}\begin{pmatrix} 1\\0\end{pmatrix}+O(\xi^3)
\end{split}
\end{equation}
form a complex-valued, orthogonal basis of the spectral subspace associated with eigenvalues 
$\omega_{0,\xi}^+$, $\omega_{0,\xi}^-$, $\omega_{1,\xi}^-$, $\omega_{-1,\xi}^-$ of $\mathcal{L}_{\xi,0}$, where
\begin{equation}\label{def:alpha}
\boldsymbol{\alpha}=\frac12\frac{k^2}{1+m^2(k)} \begin{pmatrix}{\displaystyle \frac{3m(k)m'(k)^2}{1+m^2(k)}-m''(k)}\\ 
{\displaystyle \frac{m'(k)^2(2m^2(k)-1)}{1+m^2(k)}-m(k)m''(k)} \end{pmatrix}.
\end{equation}

Observe that functions in \eqref{def:w&e} for $\xi\neq0$ vary with small values of $\xi$ to the leading order,
since the directions of the vector-valued functions vary with $\xi$ to the leading order,
and therefore functions in \eqref{E:eigen-xi} vary to the linear order. 
This may cause change in an eigenvalue near the origin of $\mathcal{L}_{\xi,a}$ at the leading order.
Hence we must take it into account when we construct basis functions 
of the spectral subspace associated with near-zero eigenvalues.
In the case of \eqref{E:bbm}, or \eqref{E:kdv}, (see the previous section, or \cite{HJ2}), in contrast, 
functions in \eqref{E:eigen0-bbm} form a basis for $\mathcal{L}_{\xi,0}$ for $\xi=0$ and all $\xi$ small.

\

For $|a|$ small and $\xi=0$, zero is a generalized $L_{2\pi}^2\times L_{2\pi}^2$-eigenvalue 
of $\mathcal{L}_{0,a}$ with algebraic multiplicity four and geometric multiplicity three, and 
\begin{equation}\label{E:eigen-a}
\begin{aligned} 
\P_{1,a}(z):=& \frac{1}{m^2(k)-1}( (\partial_{b_1}c)\partial_a \u- (\partial_{a}c)\partial_{b_1}\u)(k,a,0,0)(z) \\
&\hspace{95pt}-a\left( 1+\frac{1}{2}\frac{m^2(2k)}{m^2(k)-m^2(2k)}\right)\begin{pmatrix}0\\1 \end{pmatrix} \\
=&\begin{pmatrix} -m(k)\\ 1\end{pmatrix}\cos z
+\frac{1}{2}a\frac{m^2(2k)m^2(k)}{m^2(k)-m^2(2k)}\begin{pmatrix}1\\0 \end{pmatrix} \\
&\hspace{68pt}+a\frac{m^2(k)}{m^2(k)-m^2(2k)}\begin{pmatrix}-m^2(2k)\\m(k)\end{pmatrix}\cos 2z
+O(a^2),\hspace*{-10pt}\\
\P_{2,a}(z):=&\frac{m(k)}{a}\partial_z \u(k,a,0,0)(z) \\
=&\begin{pmatrix} -m(k)\\ 1\end{pmatrix}\sin z+a\frac{m^2(k)}{m^2(k)-m^2(2k)}\begin{pmatrix}-m^2(2k)\\m(k) \end{pmatrix}\sin 2z+O(a^2),\hspace*{-10pt}\\
\P_{3,a}(z):=& \begin{pmatrix}0\\1 \end{pmatrix},\\
\P_{4,a}(z):=&\frac{1}{m^2(k)-1}\partial_{b_2} \u(k,a,0,0)(z)+m(k) \begin{pmatrix}0\\1 \end{pmatrix}\\
=&\begin{pmatrix}1\\0 \end{pmatrix}+a \begin{pmatrix}m(k)\\-2 \end{pmatrix}\cos z+O(a^2) 
\end{aligned}
\end{equation}
form a basis of the corresponding generalized eigenspace. Indeed,
differentiating \eqref{E:quad-main} with respect to $z$, $a$, $b_1$, $b_2$, we find that
\begin{align*}
&\mathcal{L}_{0,a}(\partial_z \u)=0,\qquad
&&\mathcal{L}_{0,a}(\partial_a\u)=(\partial_a c)(\partial_z \u),\\
&\mathcal{L}_{0,a}(\partial_{b_1}\u)=(\partial_{b_1} c)(\partial_z \u),\qquad
&&\mathcal{L}_{0,a}(\partial_{b_2}\u)=(\partial_{b_2} c)(\partial_z \u),
\end{align*}
respectively. Moreover $\mathcal{L}_{0,a} \begin{pmatrix}0\\1 \end{pmatrix}= 0$.
Therefore \eqref{E:eigen-a} follows at once; see \cite[Lemma~3.1]{HJ2} for details.
In the case of $a=0$ note that $\P_{j,a}$'s, $j=1,2,3,4,$ in \eqref{E:eigen-a} reduce to functions in \eqref{E:eigen00}.

Ultimately, let 
\begin{align}
\P_1(z)=&\begin{pmatrix}-m(k)\\ 1 \end{pmatrix}\cos z 
-i\xi \frac{km'(k)}{1+m^2(k)}\begin{pmatrix} 1\\ m(k)\end{pmatrix}\sin z \label{E:bnesq-1xia} \\
&+\frac{1}{2}a\frac{m^2(2k)m^2(k)}{m^2(k)-m^2(2k)}\begin{pmatrix}1\\0 \end{pmatrix}
+a\frac{m^2(k)}{m^2(k)-m^2(2k)}\begin{pmatrix}-m^2(2k)\\m(k) \end{pmatrix}\cos 2z\hspace*{-10pt} \notag \\
&+\xi^2 \boldsymbol{\alpha} \cos z+O(\xi^3+a^2),\notag \\
\P_2(z)=&\begin{pmatrix}-m(k)\\ 1 \end{pmatrix}\sin z 
+i\xi \frac{km'(k)}{1+m^2(k)}\begin{pmatrix} 1\\ m(k)\end{pmatrix}\cos z\label{E:bnesq-2xia}\\
&+a\frac{m^2(k)}{m^2(k)-m^2(2k)}\begin{pmatrix}-m^2(2k)\\m(k) \end{pmatrix}\sin 2z
+\xi^2 \boldsymbol{\alpha} \sin z+O(\xi^3+a^2) \notag
\intertext{and}
\P_3(z)=&\begin{pmatrix}0\\ 1 \end{pmatrix}+\frac{k^2}{12}\begin{pmatrix} 0\\1\end{pmatrix}\xi^2+O(\xi^3),\label{E:bnesq-3xia}\\
\P_4(z)=&\begin{pmatrix}1\\ 0 \end{pmatrix}+a\begin{pmatrix} m(k)\\-2\end{pmatrix}\cos z
-\frac{k^2}{12}\begin{pmatrix} 1\\0\end{pmatrix}\xi^2+O(\xi^3+a^2)\label{E:bnesq-4xia},
\end{align}
where $\boldsymbol{\alpha}$ is in \eqref{def:alpha}. 
Note that $\P_j$'s, $j=1,2,3,4$, depend continuously on $\xi$ and $a$.
In the case of $a=0$, they reduce to $\P_{j,\xi}$'s, $j=1,2,3,4,$ in \eqref{E:eigen-xi},
and in the case of $\xi=0$, they reduce to $\P_{j,a}$'s, $j=1,2,3,4,$ in \eqref{E:eigen-a}.

\

To recapitulate, in the case of $\xi$ small and $a=0$, 
$\mathcal{L}_{\xi,0}$ possesses four purely imaginary eigenvalues near the origin 
and functions in \eqref{E:bnesq-1xia}-\eqref{E:bnesq-4xia}, restricted to $a=0$, 
form an orthogonal basis of the associated spectral subspace.
In the case of $\xi=0$ and $|a|$ small, moreover, $\mathcal{L}_{0,a}$ possesses four eigenvalues at the origin
and functions in \eqref{E:bnesq-1xia}-\eqref{E:bnesq-4xia}, restricted to $\xi=0$, 
form a basis of the associated eigenspace. 

In order to study how the four eigenvalues at the origin of $\L_{0,a}$ vary with $\xi$ and $|a|$ small,
we proceed as in \cite{HJ2}, or in the previous section, and compute $4\times 4$ matrices 
\begin{equation}\label{def:bnesq-B}
\mathbf{B}_{\xi,a}=\left( \frac{\l \L_{\xi,a}\P_j, \P_k\r}{\l \P_j, \P_j\r}\right)_{j,k=1,2,3,4}
\quad\text{and}\quad
\mathbf{I}_{\xi,a}=\left( \frac{\l \P_j, \P_k\r}{\l \P_j, \P_j\r}\right)_{j,k=1,2,3,4}
\end{equation}
up to the quadratic order in $\xi$ and the linear order in $a$ as $\xi, a \to 0$,
where $\P_j$'s, $j=1,2,3,4,$ are in \eqref{E:bnesq-1xia}-\eqref{E:bnesq-4xia} 
and $\langle\,,\rangle=\langle\,,\rangle_{L^2_{2\pi}\times L^2_{2\pi}}$ is in \eqref{def:prod2}.
For $\xi$ and $|a|$ small, eigenvalues of $\L_{\xi,a}$ agree in location and multiplicity 
with the roots of $\det(\mathbf{B}_{\xi,a}-\lambda\mathbf{I}_{\xi,a})=0$; 
see \cite[Section~4.3.5]{K}, for instance, for details.
Unlike in \eqref{def:BI}, note that $\mathbf{I}_{\xi,a}$ depends on $\xi$.

We begin by computing that 
\begin{align}\label{E:L-bnesq}
\mathcal{L}_{\xi,a}
=& e^{-i\xi z}\partial_z \begin{pmatrix} m(k) & \M_k^2\\ 1 & m(k)\end{pmatrix} e^{i\xi z}+2am(k) e^{-i\xi z}\partial_z \begin{pmatrix} 0&0\\1&0\end{pmatrix}\cos z\,e^{i\xi z}+O(a^2)\notag\\
=&\mathcal{L}_{0,0}+i\xi [\mathcal{L}_{0,0},z]-\frac{\xi^2}{2}[[\mathcal{L}_{0,0},z],z]\notag\\
&+2am(k)\partial_z\begin{pmatrix} 0&0\\1&0\end{pmatrix}\cos z+2i\xi am(k)\begin{pmatrix} 0&0\\1&0\end{pmatrix}\cos z+O(\xi^3+\xi^2 a+a^2)\notag\\
=&:L+2am(k)\partial_z\begin{pmatrix} 0&0\\1&0\end{pmatrix}\cos z+2i\xi am(k)\begin{pmatrix} 0&0\\1&0\end{pmatrix}\cos z+O(\xi^3+\xi^2 a+a^2)\hspace*{-10pt}
\end{align}
as $\xi,a\to 0$. The first equality uses \eqref{E:bnesq-eigen}, \eqref{E:bnesq-bloch} and \eqref{def:u}, \eqref{def:c},
and the second equality uses a Baker-Campbell-Hausdorff expansion.
Note that $L=\L_{\xi,0}$ to the second order for $\xi$ sufficiently small. 

Using \eqref{E:a=0bnesq} and \eqref{def:w&e}, or its Taylor expansion (see (M1) of Assumption~\ref{A:m}),
we make an explicit calculation and fine that 
\[
L\begin{pmatrix} \alpha \\ \beta \end{pmatrix}=i\xi\begin{pmatrix}m(k)\alpha+\beta\\ \alpha+m(k)\beta \end{pmatrix},
\]
\begin{align*}
L\begin{pmatrix} \alpha \\ \beta \end{pmatrix}\left\{ \begin{matrix} \cos z\\ \sin z\end{matrix}\right\} =& \mp (\alpha+m(k)\beta)\begin{pmatrix}m(k)\\ 1 \end{pmatrix}\left\{ \begin{matrix} \sin z\\ \cos z\end{matrix}\right\}\\
&+i\xi \begin{pmatrix}m(k)\alpha +m(k)(m(k)+2km'(k))\beta\\ \alpha+m(k)\beta \end{pmatrix}\left\{ \begin{matrix} \cos z\\ \sin z\end{matrix}\right\}\\
&\mp \xi^2 k \beta (2m(k)m'(k)+k(m'(k)^2+m(k)m''(k)))\begin{pmatrix}1\\0 \end{pmatrix}\left\{ \begin{matrix} \sin z\\ \cos z\end{matrix}\right\},
\end{align*}
\begin{align*}
L\begin{pmatrix} \alpha \\ \beta \end{pmatrix}\left\{ \begin{matrix} \cos 2z\\ \sin 2z\end{matrix}\right\} =& \mp 2\begin{pmatrix}m(k)\alpha +m^2(2k) \beta\\ \alpha+m(k)\beta \end{pmatrix}\left\{ \begin{matrix} \sin 2z\\ \cos 2z\end{matrix}\right\} \\
&+i\xi \begin{pmatrix}m(k)\alpha +m(2k)(m(2k)+4km'(2k))\beta\\ \alpha+m(k)\beta \end{pmatrix}\left\{ \begin{matrix} \cos 2z\\ \sin 2z\end{matrix}\right\} \\
&\mp 2\xi^2 k \beta (m(2k)m'(2k) 
+k(m'(2k)^2+m(2k)m''(2k)))\begin{pmatrix}1\\0 \end{pmatrix}\left\{ \begin{matrix} \sin 2z\\ \cos 2z\end{matrix}\right\}.
\end{align*}
Substituting \eqref{E:bnesq-1xia}-\eqref{E:bnesq-4xia} into \eqref{E:L-bnesq}, 
and using the above and \eqref{E:a=0bnesq}, we make a lengthy but straightforward calculation to find that
\begin{align*}
\L_{\xi,a}\P_1=&2i\xi k m(k)m'(k)\begin{pmatrix}1\\0 \end{pmatrix}\cos z
+i\xi km'(k)\begin{pmatrix} m(k)\\1\end{pmatrix}\cos z \\
&-i\xi am^2(k)\begin{pmatrix} 0\\1\end{pmatrix} (\cos 2z+1)
-2i\xi a\frac{ km(k)m'(k)}{1+m^2(k)}\begin{pmatrix}0\\1 \end{pmatrix} \cos 2z \\
&+\frac{1}{2}i\xi a\frac{m^2(k)m^2(2k)}{m^2(k)-m^2(2k)}\begin{pmatrix} m(k)\\1 \end{pmatrix}
+i\xi a\frac{m^2(k)}{m^2(k)-m^2(2k)}\begin{pmatrix} 4km(k)m(2k)m'(2k)\notag\\ m^2(k)-m^2(2k)\end{pmatrix}\cos 2z\\
&-\xi^2 k(2m(k)m'(k)+k(m'(k)^2+m(k)m''(k)))\begin{pmatrix}1\\0 \end{pmatrix} \sin z\\
&+\xi^2\frac{k m'(k)}{1+m^2(k)}\begin{pmatrix}m(k)+m^2(k)(m(k)+2km'(k))\\1+m^2(k) \end{pmatrix}\sin z\\
&-\frac12\xi^2\frac{k^2}{1+m^2(k)}(2m(k)m'(k)^2-m''(k)(1+m^2(k)))\begin{pmatrix}m(k)\\1 \end{pmatrix}\sin z\\
&+O(\xi^3+\xi^2 a+a^2),
\end{align*}
\begin{align*}
\L_{\xi,a}\P_2=&2i\xi k m(k)m'(k)\begin{pmatrix}1\\0 \end{pmatrix}\sin z
+i\xi km'(k)\begin{pmatrix} m(k)\\1\end{pmatrix}\sin z \\
&+2i\xi a\frac{ km(k)m'(k)}{1+m^2(k)}\begin{pmatrix}0\\1 \end{pmatrix} \sin 2z  \\
&+i\xi a\frac{m^2(k)}{m^2(k)-m^2(2k)}\begin{pmatrix} 4km(k)m(2k)m'(2k)\notag\\ m^2(k)-m^2(2k)\end{pmatrix}\sin 2z 
-i\xi am^2(k)\begin{pmatrix} 0\\1\end{pmatrix} \sin 2z\\
&+\xi^2 k(2m(k)m'(k)+k(m'(k)^2+m(k)m''(k)))\begin{pmatrix}1\\0 \end{pmatrix} \cos z \\
&-\xi^2\frac{ k m'(k)}{1+m^2(k)}\begin{pmatrix}m(k)+m^2(k)(m(k)+2km'(k))\\1+m^2(k) \end{pmatrix}\cos z\\
&+\xi^2\frac{ k^2}{2(1+m^2(k))}(2m(k)m'(k)^2-m''(k)(1+m^2(k)))\begin{pmatrix}m(k)\\1 \end{pmatrix}\cos z\\
&+O(\xi^3+\xi^2 a+a^2)
\end{align*}
and
\begin{align*}
\L_{\xi,a}\P_3=&i\xi \begin{pmatrix} 1\\m(k) \end{pmatrix}+O(\xi^3+\xi^2 a+a^2), \\
\L_{\xi,a}\P_4=&i\xi \begin{pmatrix} m(k)\\1 \end{pmatrix} -2am(k) \begin{pmatrix} 0\\1 \end{pmatrix} \sin z
+am(k)\begin{pmatrix} m(k)\\1\end{pmatrix} \sin z \\
&+2i\xi am(k)\begin{pmatrix} 0\\1\end{pmatrix}\cos z 
-i\xi a m(k)\begin{pmatrix}m(k)+4km'(k)\\1 \end{pmatrix}\cos z+O(\xi^3+\xi^2 a+a^2) 
\end{align*}
as $\xi, a\to 0$. Recall \eqref{def:prod2}. Using the above and \eqref{E:bnesq-1xia}-\eqref{E:bnesq-4xia},
we make another lengthy but straightforward calculation to find that 
\begin{align*}
\langle \mathcal{L}_{\xi,a}\P_1,\P_1\rangle& =\langle \mathcal{L}_{\xi,a}\phi_2,\phi_2\rangle
=-\frac{1}{2} i\xi k m'(k)(m^2(k)+1) +O(\xi^3+\xi^2 a+a^2),\\
\langle \mathcal{L}_{\xi,a}\P_1,\P_2\rangle& =-\langle \mathcal{L}_{\xi,a}\phi_2,\phi_1\rangle
=\frac{\xi^2}{2}\Big( km'(k)+\frac{k^2 m''(k)}{2}\Big)(m^2(k)+1) +O(\xi^3+\xi^2 a+a^2),\\
\langle \mathcal{L}_{\xi,a}\P_1,\P_3\rangle&=i\xi a\Big(\frac{m^2(k)m^2(2k)}{2(m^2(k)-m^2(2k))}-m^2(k) \Big) +O(\xi^3+\xi^2 a+a^2), \\
\langle \mathcal{L}_{\xi,a}\P_1,\P_4\rangle&=i\xi a\Big( \frac{m(k)^3m^2(2k)}{2(m^2(k)-m^2(2k))}+\frac{k(m^2(k)+2)m'(k)}{2}\Big)+O(\xi^3+\xi^2 a+a^2)
\intertext{and}
\langle \mathcal{L}_{\xi,a}\P_2,\P_3\rangle&=\langle \mathcal{L}_{\xi,a}\phi_2,\phi_4\rangle
=0 +O(\xi^3+\xi^2 a+a^2),\\
\langle \mathcal{L}_{\xi,a}\P_3,\P_1\rangle&=\frac{1}{2}i\xi a\frac{m^2(k)m^2(2k)}{m^2(k)-m^2(2k)}+O(\xi^3+\xi^2 a+a^2),  \\
\langle \mathcal{L}_{\xi,a}\P_3,\P_2\rangle&=  0+O(\xi^3+\xi^2 a+a^2),  \\
\langle \mathcal{L}_{\xi,a}\P_3,\P_3\rangle&=\langle \mathcal{L}_{\xi,a}\phi_4,\phi_4\rangle
= i\xi m(k)+O(\xi^3+\xi^2 a+a^2)  ,\\
\langle \mathcal{L}_{\xi,a}\P_3,\P_4\rangle&=\langle \mathcal{L}_{\xi,a}\P_4,\P_3\rangle
=  i\xi+O(\xi^3+\xi^2 a+a^2) ,\\
\langle \mathcal{L}_{\xi,a}\P_4,\P_1\rangle&=i\xi a m(k)\Big(\frac{3}{2}+\frac{m(k)^4}{2(m^2(k)-m^2(2k))}+2km(k)m'(k) \Big) +O(\xi^3+\xi^2 a+a^2),  \\
\langle \mathcal{L}_{\xi,a}\P_4,\P_2\rangle&=-\frac{a}{2}m(k)(m^2(k)+1) +O(\xi^3+\xi^2 a+a^2)   
\end{align*}
as $\xi, a\to 0$. Moreover we use \eqref{E:bnesq-1xia}-\eqref{E:bnesq-4xia} to compute that 
\begin{align*}
\langle\P_1,\P_1\rangle& =\langle \P_2,\P_2 \rangle= \frac{m^2(k)+1}{2}+O(\xi^3+\xi^2 a+a^2),\\
\langle\P_1,\P_2\rangle&=\langle \P_1,\P_3 \rangle= 0+O(\xi^3+\xi^2 a+a^2) ,\\
\langle\P_1,\P_4\rangle&=\frac{a}{2}\Big(\frac{m^2(k)m^2(2k)}{m^2(k)-m^2(2k)}-m^2(k)-2\Big)+O(\xi^3+\xi^2 a+a^2)
\intertext{and}
\langle \P_2,\P_3 \rangle&= 0+O(\xi^3+\xi^2 a+a^2),  \\
\langle \P_2,\P_4 \rangle&= -\frac{1}{2}i\xi a\frac{km(k)m'(k)}{m^2(k)+1}+O(\xi^3+\xi^2 a+a^2),  \\
\langle \P_3,\P_3 \rangle&=\langle \P_4,\P_4 \rangle= 1+O(\xi^3+\xi^2 a+a^2),\\
\langle \P_3,\P_4 \rangle&=0 +O(\xi^3+\xi^2 a+a^2) 
\end{align*}
as $\xi, a\to 0$. To summarize, \eqref{def:bnesq-B} becomes
\begin{align}
\mathbf{B}_{\xi,a} =&-\frac{a}{2} m(k)(m^2(k)+1)\begin{pmatrix} 0&0&0&0\\0&0&0&0\\0&0&0&0\\0&1&0&0\end{pmatrix}\label{E:B-bnesq}\\
&+i\xi \begin{pmatrix} -km'(k)&0&0&0\\ 0&-km'(k)&0&0\\ 0&0&m(k)&1\\0&0&1&m(k)\end{pmatrix}\notag \\
&+\frac{2i\xi a}{m^2(k)+1}\begin{pmatrix} 0&0&U_2-m^2(k)&
2m(k)U_2+\frac12 km'(k)(m^2(k)+2)\\ 0&0&0&0\\0&0&0&0\\0&0&0&0 \end{pmatrix} \notag \\
&+i\xi a \begin{pmatrix}0&0&0&0\\0&0&0&0\\U_2&0&0&0\\m(k)\left( 
\frac12(m^2(k)+3)+2U_2+2km(k)m'(k)\right)&0&0&0\end{pmatrix}\notag\\
&+\xi^2(km'(k)+\frac12k^2m''(k))\begin{pmatrix}0&1&0&0\\ -1&0&0&0\\ 0&0&0&0\\0&0&0&0 \end{pmatrix}+O(\xi^3+\xi^2 a+a^2)\notag
\end{align}
and
\begin{align}\label{E:I-bnesq}
\mathbf{I}_{\xi,a} =&\mathbf{I}+\frac12a\frac{2U_2-m^2(k)-2}{m^2(k)+1}\begin{pmatrix}0&0&0&2\\0&0&0&0\\0&0&0&0\\m^2(k)+1&0&0&0 \end{pmatrix} \\
&-i\xi a\frac12\frac{ k m(k)m'(k)}{(m^2(k)+1)^2}\begin{pmatrix} 0&0&0&0\\0&0&0&2\\0&0&0&0\\0&m^2(k)+1&0&0\end{pmatrix} +O(\xi^3+\xi^2 a+a^2)\notag
\end{align}
as $\xi, a\to 0$. Here $\mathbf{I}$ denotes the $4\times4$ identity matrix.

\subsection{Proof of Theorem \ref{thm:bnesq}}
We turn the attention to the roots of the characteristic polynomial
\[
\det(\mathbf{B}_{\xi,a}-\lambda \mathbf{I}_{\xi,a})=
D_4(\xi,a)\lambda^4+iD_3(\xi,a)\lambda^3+D_2(\xi,a)\lambda^2+iD_1(\xi,a)\lambda+D_0(\xi,a)
\]
for $\xi$ and $|a|$ sufficiently small, 
where $\mathbf{B}_{\xi,a}$ and $\mathbf{I}_{\xi,a}$ are in \eqref{E:B-bnesq} and \eqref{E:I-bnesq}.
Details are similar to those in \cite[Section~3.3]{HJ2}, or in the previous section.
Hence we merely hit the main points.

Note that $D_j$'s, $j=0,1,2,3,4,$ are all real and depend smoothly on $\xi$ and $a$ for $\xi,|a|$ sufficiently small.
Note moreover that $D_1,D_3$ are even in $\xi$ and $D_0,D_2,D_4$ are odd
and that $D_j$'s, $j=0,1,2,3,4,$ are even in $a$. 
Since $\lambda=0$ is a root with multiplicity four for $\xi=0$ and $|a|$ sufficiently small,
$D_j=\xi^{4-j}d_j$, $j=0,1,2,3,4,$ for some $d_j$'s, real and smooth in $\xi$, $a$ and even in $a$. 
We may therefore write that
\[
\det(\mathbf{B}_{\xi,a}-(-i\xi)\lambda \mathbf{I}_{\xi,a})= 
\xi^4 (d_4(\xi,a)\lambda^4-d_3(\xi,a)\lambda^3-d_2(\xi,a)\lambda^2+d_1(\xi,a)\lambda+d_0(\xi,a)).
\]
The underlying, periodic traveling wave of \eqref{E:bnesq}, and hence \eqref{E:main}, is then modulationally unstable,
provided that $\det(\mathbf{B}_{\xi,a}-(-i\xi)\lambda \mathbf{I}_{\xi,a})$ admits a pair of complex roots, 
or equivalently
\begin{align*}
\text{disc}_{\text{Bnesq}}(\xi,a):=
&256d_4^3d_0^3-192d_4^2d_3d_1d_0^2-128d_4^2d_2^2d_0^2+144d_4^2d_2d_1^2d_0-27d_4^2d_1^4 \\
&+144d_4d_3^2d_2d_0^2-6d_4d_3^2d_1^2d_0-80d_4d_3d_2^2d_1d_0+18d_4d_3d_2d_1^3+16d_4d_2^4d_0\\
&-4d_4d_2^3d_1^2-27d_3^4d_0^2+18d_3^3d_2d_1d_0-4d_3^3d_1^3-4d_3^2d_2^3d_0+d_3^2d_2^2d_1^2<0
\end{align*}
for $\xi$ and $|a|$ sufficiently small. We may furthermore write that 
\[
\text{disc}_{\text{Bnesq}}(\xi,a)=\text{disc}_{\text{Bnesq}}(k;\xi,0)+\text{ind}_{\text{Bnesq}}(k)a^2+O(a^2(a^2+\xi^2))
\]
as $\xi,a\to0$. Since  \eqref{E:B-bnesq} and \eqref{E:I-bnesq} become
\begin{align*}
\mathbf{B}_{\xi,0}=&i\xi \begin{pmatrix} -km'(k)&0&0&0\\0&-km'(k)&0&0\\0&0&m(k)&1\\0&0&1&m(k)\end{pmatrix}\\
&+\xi^2( km'(k)+\tfrac12k^2m''(k))\begin{pmatrix} 0&1&0&0\\-1&0&0&0\\0&0&0&0\\0&0&0&0\end{pmatrix}+O(\xi^3)
\end{align*}
and $\mathbf{I}_{\xi,0}=\mathbf{I}$ for $\xi$ sufficiently small, $\text{disc}_{\text{Bnesq}}(k;\xi,0)>0$ for all $k>0$.
Therefore if $\text{ind}_{\text{Bnesq}}(k)<0$ then $\text{disc}_{\text{Bnesq}}(k;\xi,a)<0$ 
for $\xi$ sufficiently small, depending on $a$ sufficiently small, implying modulational instability. 
A Mathematica calculation reveals that the sign of $\text{ind}_{\text{Bnesq}}(k)$ agrees with 
that of \eqref{def:ind-bnesq}. This complete the proof.

In case $\text{ind}_{\text{Bnesq}}(k)>0$, on the other hand, 
$\text{disc}_{\text{Bnesq}}(k;\xi,a)>0$ for all $k>0$ and $\xi,|a|$ sufficiently small,
whence $\mathbf{B}_{\xi,a}-(-i\xi)\lambda \mathbf{I}_{\xi,a}$ will admit 
either four real eigenvalues or four complex eigenvalues. 
We discuss in Appendix~\ref{sec:disc} the classification of the roots of a quartic polynomial,
 which will help to completely determine the modulational stability and instability.

\section{Applications}\label{sec:applications}
We illustrate the results in Section~\ref{sec:bbm} and Section~\ref{sec:bnesq} by discussing some examples.

\subsection{The Benjamin-Bona-Mahony equation}\label{sec:bbm1}
Note that 
\[
m(k)=\frac{1}{1+k^2}
\]
satisfies Assumption~\ref{A:m} and it reduces \eqref{E:bbm} to the BBM equation (see \eqref{E:bbm1}). 
For an arbitrary $k>0$, note from Lemma~\ref{lem:existence} that
\begin{equation}\label{E:soln-bbm1}
\left\{\begin{split}
&u(x,t;k,a)=a\cos(k(x-ct))+a^2\frac{1+k^2}{6k^2}(\cos(2k(x-ct))-3)+O(a^3),\hspace*{-10pt}\\
&c(k,a)=\frac{1}{1+k^2}-a^2\frac{5}{6k^2}+O(a^4),
\end{split}\right.
\end{equation}
for $|a|\ll1$, make a sufficiently small, $2\pi/k$-periodic wave of the BBM equation
traveling at the speed $c(k,a)$. Note that $c(k,a)<1$ for all $k>0$ and $|a|$ sufficiently small.
We pause to remark that another family of periodic traveling waves of the BBM equation with sufficiently small amplitudes
was constructed in \cite{Haragus08} near $c-1$, for which $c>1$ and $k<1$. 

A straightforward calculation reveals that 
\[
i_1(k)=\frac{2k(k^2-3)}{(1+k^2)^3}>0
\]
if and only if $k>\sqrt{3}$, 
\[
i_2^- (k)=-\frac{k^2(3+k^2)}{(1+k^2)^2}<0\quad\text{and}\quad i_3^- (k)=\frac{3k^2}{1+5k^2+4k^4}>0 
\]
for all $k>0$, where $i_1, i_2^-, i_3^-$ are in \eqref{def:i123}. Moreover 
\[
i_{\text{BBM}}(k)=\frac{k^2(3+5k^2)}{(1+k^2)^2(1+4k^2)}>0 
\]
for all $k>0$, where $i_{\text{BBM}}$ is in \eqref{def:i-bbm}.
Collectively, $\text{ind}_{\text{BBM}}(k)<0$ if and only if $k>\sqrt{3}$,
where $\text{ind}_{\text{BBM}}$ is in \eqref{def:ind-bbm}.
It then follows from Theorem~\ref{thm:bbm} that \eqref{E:soln-bbm1} is modulationally unstable if $k>\sqrt{3}$
and it is stable in the vicinity of the origin in the spectral plane, otherwise. 

Away from the origin in the spectral plane, since the $L^2_{2\pi}$-spectrum of $\mathcal{L}_{\xi,a}$ 
associated with \eqref{E:bbm1} is symmetric about the imaginary axis, 
its eigenvalues may leave the imaginary axis, leading to instability, as $\xi$ and $a$ vary, 
only through collisions with other purely imaginary eigenvalues. 
Recall \eqref{E:a=0bbm} and \eqref{def:w-bbm}. Since $m(k)$ decreases in $k$, we deduce that
\[ 
\dots <\omega_{-3,\xi}<\omega_{-2,\xi}<0<\omega_{1,\xi}<\omega_{2,\xi}<\omega_{3,\xi}<\dots
\]
for each $\xi\in[0,1/2]$. 
Moreover it is readily seen that $\omega_{0,\xi}<0$ and $\omega_{-1,\xi}>0$ for all $\xi\in[0,1/2]$. 
A straightforward calculation reveals that if $\omega_{-1,\xi}$ and $\omega_{n,\xi}$ collide 
for some $n\geq 1$ an integer and $\xi\in[0,1/2]$ then $n=1$, whence 
\[ 
k=\sqrt{\frac{3}{1-\xi^2}}\geq \sqrt{3}.
\]
But the underlying wave is modulationally unstable in the range.
Similarly if $\omega_{0,\xi}$ and $\omega_{n,\xi}$ collide for some $n\leq -2$ an integer and $\xi\in[0,1/2]$ then
\[ 
k=\sqrt{\frac{1-n^2+3n\xi-3\xi^2}{\xi^4-2n\xi^3+(n^2+1)\xi^2-n\xi}}\geq 2\sqrt{3/5}.
\]
For $k<2\sqrt{3/5}$, therefore, eigenvalue collide only at the origin, 
which incidentally does not lead to instability since $\text{ind}_{\text{BBM}}(k)>0$.
In other words, the underlying wave is spectrally stable. Below we summarize the conclusion.

\begin{corollary}[Modulational instability vs. spectral stability for the BBM equation]\label{cor:bbm1}
A sufficiently small, $2\pi/k$-periodic traveling wave of \eqref{E:bbm1} 
is spectrally unstable to long wavelength perturbations if $k>\sqrt{3}$,
and it is spectrally stable to square integrable perturbations if $0<k<2\sqrt{3/5}$.
\end{corollary}

For $2\sqrt{3/5}<k\leq \sqrt{3}$, one may make a Krein signature calculation to study the stability and instability.
But we do not pursue here.

Corollary~\ref{cor:bbm1} agrees with that in \cite{J2010}, 
where the author proved that periodic traveling waves of the BBM equation of sufficiently large periods,
or conversely sufficiently small wave numbers, (but not necessarily small amplitudes)
are modulationally stable. As a matter of fact, periodic traveling waves of the BBM equation 
are expected to tend to solitary waves as their period increases to infinity.

Moreover Corollary~\ref{cor:bbm1} complements the result in \cite{Haragus08},
where the author employed a similar method to show that 
a periodic traveling wave of \eqref{E:bbm1} with sufficiently small amplitude near $c-1$
is modulationally stable. 

\subsection{The regularized Boussinesq equation}\label{sec:bnesq1}
Note that 
\[
m(k)=\sqrt{\frac{1}{1+k^2}}
\]
satisfies Assumption~\ref{A:m} and it reduces \eqref{E:bnesq} to the regularized Boussinesq equation (see \eqref{E:bnesq1}). 
For an arbitrary $k>0$, note from Lemma~\ref{lem:exist-bnesq} that 
\begin{equation}\label{E:soln-bnesq1}
\left\{\begin{split}
&u(x,t;k,a)=a\frac{1}{\sqrt{1+k^2}}\cos(k(x-ct))+a^2\frac{1}{6k^2}(\cos(2k(x-ct))-3)+O(a^3),\hspace*{-20pt}  \\
&c(k,a)=\frac{1}{\sqrt{1+k^2}}\Big(1-\frac{5}{12k^2}a^2\Big)+O(a^4),
\end{split}\right.
\end{equation}
for $|a|\ll1$, make a sufficiently small, $2\pi/k$-periodic wave of \eqref{E:bnesq1} traveling at the speed $c(k,a)$. 
A straightforward calculation reveals that 
\[
i_1(k)=-\frac{3k}{(1+k^2)^{5/2}}<0
\]
for all $k>0$, 
\[
i_2^- i_2^+ (k)=\frac{1-(1+k^2)^3}{(1+k^2)^3}<0\quad\text{and}\quad
 i_3^- i_3^+ (k)= \frac{1}{1+k^2}-\frac{1}{1+4k^2}>0
\]
for all $k>0$, where $i_1,i_2^-,i_3^-$ are in \eqref{def:i123} and $i_2^+,i_3^+$ are in \eqref{def:i+}. Moreover
\[
i_{\text{Bnesq}}(k)=\frac{k^2(5k^6+14k^4+12k^2+2)}{(1+k^2)^4(1+4k^2)}>0 
\]
for all $k>0$, where $i_{\text{Bnesq}}$ is in \eqref{def:i-bnesq}. 
Collectively, $\text{ind}_{\text{Bnesq}}(k)>0$ for all $k>0$, 
where $\text{ind}_{\text{Bnesq}}$ is in \eqref{def:ind-bnesq}. 
This is inconclusive since if the discriminant of the quartic polynomial
$\text{disc}_{\text{Bnesq}}(\mathbf{B}_{\xi,a}-(-i\xi)\lambda\mathbf{I}_{\xi,a})>0$ for all $k>0$, 
where $\mathbf{B}_{\xi,a}$ and $\mathbf{I}_{\xi,a}$ are in \eqref{E:B-bnesq} and \eqref{E:I-bnesq}, 
then $P(\lambda;k,\xi,a):=\det(\mathbf{B}_{\xi,a}-(-i\xi)\lambda\mathbf{I}_{\xi,a})$
possesses either four real roots, implying stability, or two pairs of complex roots, implying instability; 
see the previous section for details.

In order to determine the nature of the roots of the quartic characteristic polynomial,
we calculate additional discriminants \eqref{def:disc1} and \eqref{def:disc2} in Theorem~\ref{thm:quartic}.
A Mathematica calculation reveals that 
\begin{align*}
\text{disc}_1(P)=&-4(2+(m(k)+km'(k))^2)+O(\xi^2+a^2)<0, \\
\text{disc}_2(P)=&-16(1+2(m(k)+km'(k))^2)+O(\xi^2+a^2)<0
\end{align*}
as $\xi,a\to0$. 
Therefore it follows from Theorem~\ref{thm:quartic} that $P=\det(\mathbf{B}_{\xi,a}-(-i\xi)\lambda\mathbf{I}_{\xi,a})$
possesses four real roots for all $k>0$ for $\xi$ and $|a|$ sufficiently small. 
In other words, the underlying wave is modulationally stable.
Below we summarize the conclusion.

\begin{corollary}[Modulational stability for regularized Boussinesq equation]\label{cor:bnesq1}
A sufficiently small, $2\pi/k$-periodic traveling wave of \eqref{E:bnesq1} 
is stable to square integrable perturbations in the vicinity of the origin in the spectral plane.
\end{corollary}
 
Here we do not study collision of eigenvalues away from the origin.

\subsection{Equations with fractional dispersion}\label{sec:fdispersion}
Note that 
\[
m(k)=1+|k|^\alpha
\]
satisfies Assumption~\ref{A:m} if $\alpha\geq 2$\footnote{
Assumption (M1) dictates that $\alpha\geq 2$. The authors believe that one may relax the regularity requirement, but here we are not interested in achieving a minimal regularity requirement.}
 and it reduces \eqref{E:kdv}, \eqref{E:bbm} and \eqref{E:bnesq} to
\begin{equation}\label{E:fkdv}
u_t+(1+\Lambda^{\alpha})u_x+(u^2)_x=0,
\end{equation}
\begin{equation}\label{E:fbbm}
u_t+(1+\Lambda^{\alpha})(u+u^2)_x=0
\end{equation}
and
\begin{equation}\label{E:fbnesq}
u_{tt}+(1+\Lambda^{\alpha})^2(u+u^2)_{xx}=0,
\end{equation} 
respectively, where $\Lambda=\sqrt{-\partial_x^2}$ is defined via the Fourier transform as 
\[
\widehat{\Lambda f}(k)=|k|\widehat{f}(k).
\] 
In the case of $\alpha=2$, note that \eqref{E:fkdv} corresponds to the KdV equation (see \eqref{E:kdv1}), 
and in the case of $\alpha=1$ the Benjamin-Ono equation. 
In the case of\footnote{Note that $\Lambda^\alpha\partial_x$ is not singular if $\alpha\geq -1$.} 
$\alpha=-1/2$, moreover, \eqref{E:fkdv} was argued in \cite{Hur-blowup}
to have relevance to water waves in two dimensions in the infinite depth.

Since \eqref{E:fkdv}, \eqref{E:fbbm} and \eqref{E:fbnesq} share the dispersion relation in common, 
$(i_1i_2^-i_3^-)(k)$ in \eqref{def:ind-kdv} and \eqref{def:ind-bbm}, 
and $(i_1i_2^-i_2^+i_3^-i_3^+)(k)$ in \eqref{def:ind-bnesq}
enjoy the same sign. As a matter of fact, they are positive for all $k>0$. 
Consequently, the modulational stability and instability of a sufficiently small, periodic traveling wave of 
\eqref{E:fkdv}, \eqref{E:fbbm} or \eqref{E:fbnesq} is determined by the sign of 
(see \eqref{def:i-kdv}, \eqref{def:i-bbm} or \eqref{def:i-bnesq})
\begin{align*}
i_{\text{KdV}}(k)=&3-2^{1+\alpha}+\alpha,\\
i_{\text{BBM}}(k)=&3 - 2^{(1 + \alpha)} + \alpha + 2^{\alpha} (1 + \alpha) k^{\alpha},\\
i_{\text{Bnesq}}(k)=&2(1+k^{\alpha})^2+(1+2^{\alpha}k^{\alpha})^2(-3+(1+k^{\alpha})^2(1+(1+\alpha)k^{\alpha})^2),
\end{align*}
respectively.

Note that $i_{\text{KdV}}$ is independent of $k$ and it is negative if $\alpha<1$,
implying modulationally instability, whereas it is positive if $\alpha>1$, implying modulational stability. 
The result agrees with that in \cite{J2013}, 
which requires that the dispersion symbol be merely once continuously differentiable. 

\begin{figure}[h] 
\includegraphics[scale=0.7]{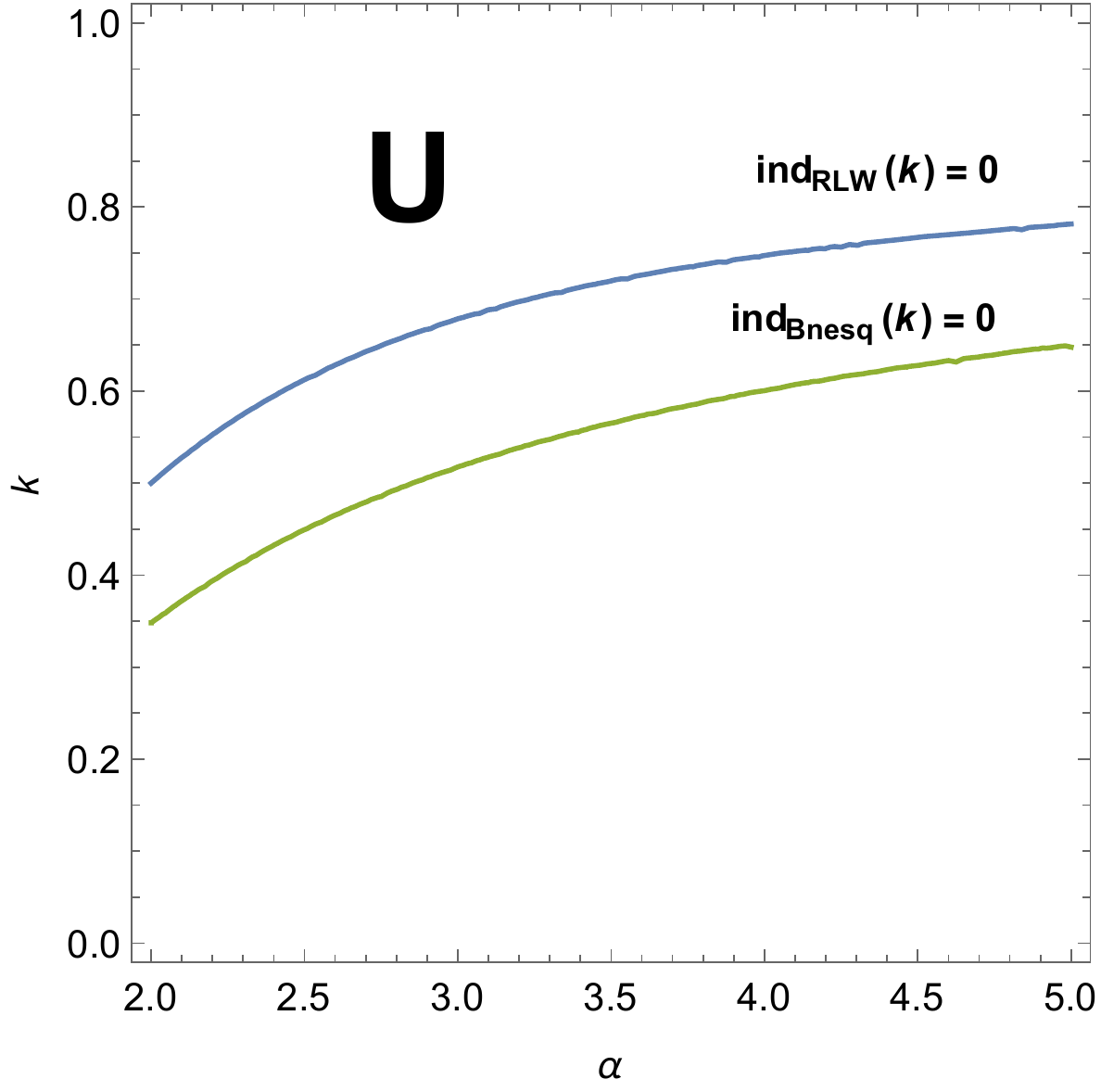}\quad
\caption{Stability diagram in the $\alpha$ versus $k$ plane for sufficiently small periodic traveling waves 
of \eqref{E:fbbm} and \eqref{E:fbnesq} with fractional dispersion. 
Above the level curves, where $ind_{BBM}(k)=0$ and $ind_{Bnesq}(k)=0$ are the regions of modulational instability.}
\end{figure}\label{f:fracdisp}

Figure~1 illustrates regions of instability for \eqref{E:fbbm} and \eqref{E:fbnesq}
in the range $\alpha>2$, where all solutions of \eqref{E:fkdv} are stable. We deduce that 
for each $\alpha>2$, a sufficiently small, periodic traveling wave of \eqref{E:fbbm} or \eqref{E:fbnesq}
is modulationally unstable if the wave number is greater than a critical value. 
The critical wave number for \eqref{E:fbbm} is larger than that for \eqref{E:fbnesq}, 
implying that the nonlinear effects of \eqref{E:fbnesq} are stronger.

\subsection*{Acknowledgements}
The authors thank Mariana Haragus for helpful discussions.
VMH is supported by the National Science Foundation CAREER DMS-1352597, 
an Alfred P. Sloan research fellowship, and an Arnold O. Beckman research award RB14100, 
a Beckman fellowship of Center for Advanced Study at the University of Illinois at Urbana-Champaign. 
AKP is supported through an Arnold O. Beckman research award RB14100 at the University of Illinois at Urbana-Champaign.

\begin{appendix}

\section{Proof of Lemma~\ref{lem:existence}}\label{sec:existence}
The proof follows along the same line as the arguments in \cite[Appendix~A]{J2013}, for instance.
We may assume that $\alpha<0$ in (M3) in Assumption~\ref{A:m}. Let
\[ 
F(u;k,c,b)=\M_k (u+u^2)-cu-(c-1)^2b
\]
and note from (M3) of Assumption~\ref{A:m} and the Sobolev inequality that 
$F:H^{1}_{2\pi}\times \mathbb{R}_+\times \mathbb{R}_+\times \mathbb{R} \to H^1_{2\pi}$ is well defined.  
In the case of $\alpha>0$ in (M3) of Assumption~\ref{A:m}, let
\[ F(u;k,c,b)=u+u^2-c\M_k^{-1}u-(c-1)^2b,\]
instead, and the proof is nearly identical. Note that 
\[
\partial_uF(u;k,c,b)v=(\M_k(1+2u)-c)v \in L^2_{2\pi}, \qquad v\in H^{1}_{2\pi},
\] 
and $\partial_k F(u;k,c,b)\delta:=\M'_\delta (u+u^2)$, $\delta \in \mathbb{R}$, are continuous, 
where a straightforward calculation reveals that
\[
\M'_\delta e^{inz}=\delta n m'(k n)e^{inz} \quad \text{for}\quad n\in\mathbb{Z}.
\] 
Since 
\[
\partial_cF(u;k,c,b)=-u+2(c-1)b\quad\text{and}\quad \partial_bF(u;k,c,b)=-(c-1)^2
\] 
are continuous, we deduce that 
$F:H^{1}_{2\pi} \times \mathbb{R}_+\times \mathbb{R}_+\times \mathbb{R} \to H^1_{2\pi}$ is $C^1$. 
Recall that $u_0$ and $c_0$, in \eqref{def:u0} and \eqref{def:c0}, satisfy that 
\[
Le^{\pm iz}=:(\M_k(1+2u_0)-c_0)e^{\pm iz}=0.
\] 

For arbitrary $k>0$ and $|b|$ sufficiently small, we seek a non-constant solution $u\in H^{1}_{2\pi}$ near $u_0$ of
\begin{equation}\label{E:F}
F(u;k,c,b)=0
\end{equation}
for some $c$ near $c_0$. Let 
\[ 
u(z)=u_0(k,b)+\frac{1}{2}ae^{iz}+\frac{1}{2}\bar{a}e^{-iz}+v(z)\quad\text{and}\quad c=c_0+r,
\]
where $a\in \mathbb{C}$ and $v\in H_{2\pi}^{1}$ satisfying that
\[ 
\int_{-\pi}^\pi v(z)e^{\pm iz}~dz=0,
\]
and $r\in \mathbb{R}$. Substituting these into \eqref{E:F} and using $Le^{\pm iz}=0$, we arrive at that
\begin{equation}\label{ldef}
Lv=:g(a,\bar{a},v,r,b),
\end{equation}
where $g$ is analytic in its argument and $g(0,0,0,r,b)=0$ for all $r,b\in \mathbb{R}$. 
We define $\Pi:L_{2\pi}^2\rightarrow \text{ker}L$ as
\[ \Pi f(z)=\widehat{f}(1)e^{iz}+\widehat{f}(-1)e^{-iz}.\]
Since $\Pi v=0$, we may write \eqref{ldef} as
\begin{equation}\label{lpdef}
Lv=(I-\Pi)g(a,\bar{a},v,r,b)\quad \text{and} \quad 0=\Pi g(a,\bar{a},v,r,b).
\end{equation}
Note that 
\[ 
( L_{|(I-\Pi)H_{2\pi}^{1}})^{-1}f(z)=\sum_{n\neq \pm 1}\frac{\widehat{f}(n)}{(1+2u_0)m(kn)-c_0}e^{inz}.\]
Consequently, we may rewrite \eqref{lpdef} as
\begin{equation}\label{linpdef}
v=L^{-1}(I-\Pi)g(a,\bar{a},v,r,b)\quad \text{and}\quad 0=\Pi g(a,\bar{a},v,r,b).
\end{equation}
Clearly $( L_{|(I-\Pi)H_{2\pi}^{1}})^{-1}$ depends analytically on its arguments. 

It follows from the implicit function theorem that a unique solution
\[ v=V(a,\bar{a},r,b)\in (I-\Pi)H_{2\pi}^{1}\]
exists to the former equation in \eqref{linpdef} in the vicinity of $(a,\bar{a},r,b)=(0,0,0,b)$,
which depends analytically on its argument. By uniqueness, moreover,
\begin{equation}\label{Vdef}
V(0,0,r,b)=0 \qquad\text{for all $r\in \mathbb{R}$ and $|b|$ sufficiently small.}
\end{equation}
Since \eqref{E:quad-bbm} remains invariant under $z\rightarrow z+z_0$ and $z\rightarrow -z$, it follows that
\begin{align}
V(a,\bar{a},r,b)(z+z_0)=V(ae^{iz_0},\bar{a}e^{-iz_0},r,b)\quad\text{and}\quad 
V(a,\bar{a},r,b)(-z)=V(a,\bar{a},r,b)(z) \label{Vprop}
\end{align}
for any $z_0\in\mathbb{R}$.
To proceed, we rewrite the latter equation in \eqref{linpdef} as 
\[
\Pi g(a,\bar{a}, V(a,\bar{a},r,b),r,b)=0,
\]
which is solvable provided that
\begin{equation*}
Q_{\pm}(a,\bar{a},r,b):=\int_{-\pi}^\pi \frac{1}{2}(ae^{iz}\pm \bar{a}e^{-iz})g(a,\bar{a},V(a,\bar{a},r,b),r,b)~dz=0.
\end{equation*}
Taking $z_0=-2\arg (a)$ in \eqref{Vprop} we find that 
\[
Q_- (\bar{a},a,r,b)=Q_- (a,\bar{a},r,b)=-Q_- (\bar{a},a,r,b).
\]
Therefore $Q_- (a,\bar{a},r,b)=0$, which is trivial. Taking $z_0=-\arg (a)$ in \eqref{Vprop}, similarly, 
\[
Q_+ (a,\bar{a},r,b)=Q_+ (|a|,|a|,r,b).
\] 
Therefore $Q_+ (a,a,r,b)=0$ for any $a\in \mathbb{R}$. 
Since \eqref{Vdef} implies that $a^{-1}V(a,a,r,b)$ is analytic in $a$ for $|a|$ sufficiently small,
 we arrive at that
\begin{align*}
Q_+ (a,a,r,b)= \int_{-\pi}^\pi a(\cos z) g(a,\bar{a},V(a,\bar{a},r,b)(z),r,b)~dz=: a^2(\pi r+R(a,r,b)),
\end{align*}
where $R$ is analytic in its argument, even in $a$ and $R(0,0,b)=\partial_r R(0,0,b)=0$. 
It then follows from the implicit function theorem that a unique solution to
\[ \pi r(a,b)+R(a,r(a,b),b)=0\]
exists for $|a|$ sufficiently small, which is real analytic for $|a|$ sufficiently small and even in $a$.
To summarize,
\[ 
(v,r)=(V(a,\bar{a},r,b), r(|a|,b))
\]
uniquely solve \eqref{linpdef} for $|a|$, $|b|$ sufficiently small. Consequently,
\[ 
u(z)=u_0+a\cos z+V(a,a,r(|a|,b),b)(z)\quad\text{and}\quad c=c_0+r(|a|,b)
\]
solve \eqref{E:F} for $|a|$, $|b|$ sufficiently small. 

\

It remains to show \eqref{E:u(k,a,b)} and \eqref{E:c(k,a,b)}. 
Let $k>0$ be fixed and suppressed to simplify the exposition. We assume that $b=0$.
Since $u$ and $c$ depend analytically on $a$ for $|a|,|b|$ sufficiently small 
and since $c$ is even in $a$, we write that
\begin{align*}
u(k,a,b)(z):=& u_0(k,b)+ a \cos z+a^2u_2(z)+a^3u_3(z)+O(a^4) 
\intertext{and} 
c(k,a,b):=&c_0(k,b)+a^2c_2+O(a^4)
\end{align*}
as $a\to 0$, where $u_2$, $u_3,\dots$ are even and $2\pi$-periodic in $z$. 
Substituting these into \eqref{E:quad-bbm}, at the order of $a^2$, we gather that
\[
\M_k (u_2(z)+\cos^2z)-m(k)u_2(z)=0.
\]
A straightforward calculation then reveals that 
\[
u_2(z)=\frac12\Big(\frac{1}{m(k)-1}+\frac{m(2k)\cos(2z)}{m(k)-m(2k)}\Big).
\] 
Continuing, at the order of $a^3$, 
\[
\M_k (u_3(z)+2u_2(z)\cos z)-m(k)u_3(z)-c_2\cos z=0,
\]
whence 
\[
c_2=m(k)\Big(\frac{1}{m(k)-1}+\frac12\frac{m(2k)}{m(k)-m(2k)}\Big).
\] 
This completes the proof.

\section{Classification of roots of a quartic polynomial}\label{sec:disc}
Here we classify the roots of a quartic polynomial. The presentation is adopted from \cite{Rees}.

\begin{theorem}\label{thm:quartic}
Let
\[ 
P(x)=p_4x^4+p_3x^3+p_2x^2+p_1x+p_0, \qquad a\neq 0,
\]
and $x_1,x_2,x_3,x_4$ be roots of $P$. Let 
\begin{align*}
\text{disc}(P)
=&256 p_4^3p_0^3-192p_4^2p_3p_1p_0^2-128p_4^2p_2^2p_0^2+144p_4^2p_2p_1^2p_0-27p_4^2p_1^4\\ 
&+144p_4p_3^2p_2p_0^2-6p_4p_3^2p_1^2p_0-80p_4p_3p_2^2p_1p_0+18p_4p_3p_2p_1^3+16p_4p_2^4p_0 \\
&-4p_4p_2^3p_1^2-27p_3^4p_0^2+18p_3^3p_2p_1p_0-4p_3^3p_1^3-4p_3^2p_2^3p_0+p_3^2p_2^2p_1^2.
\end{align*}
Define
\begin{align}
\text{disc}_1(P)&=8p_4p_2-3p_3^2 \label{def:disc1}, \\
\text{disc}_2(P)&=64p_4^3p_0-16p_4^2p_2^2+16p_4p_3^2p_2-16p_4^2p_3p_1-3p_3^4.\label{def:disc2}
\end{align}
Then
\begin{enumerate}
\item If $\text{disc}(P)<0$ then $P$ has two real roots and two complex conjugate roots,
\item If $\text{disc}(P)>0,\text{disc}_1(P)<0$ and $\text{disc}_2(P)<0$ then $P$ has four real roots,
\item If either $\text{disc}(P)>0$ and $\text{disc}_1(P)>0$ or $\text{disc}(P)>0$ and $\text{disc}_2(P)>0$ then $P$ has two pairs of complex conjugate roots.
\end{enumerate}
\end{theorem}

\begin{proof}
If $P$ has four real roots or two pairs of complex conjugate roots then
\[
\text{disc}(P)=p_4^6 \prod_{j<k}(x_j-x_k)^2\geq 0, 
\]
which implies (1). In what follows, therefore, we assume that $\text{disc}(P)>0$. 
We make the change of variables $x\mapsto x-p_3/4p_4$ to arrive at that 
\[ 
Q(x)=x^4+q_2x^2+q_1x+q_0,
\]
where
\begin{align*}
q_2&=\frac{8p_4p_2-3p_3^2}{8p_4^2}=\frac{\text{disc}_1(P)}{8p_4^2},\\
q_1&=\frac{8p_4^2p_1-4p_4p_3p_2+p_3^3}{8p_4^3},\\
q_0&=\frac{256p_4^3p_0-64p_4^2p_3p_1+16p_4p_3^2p_2-3p_4^4}{256p_4^4}.\\
\end{align*} 
Note that 
\[ 
q_0-\frac{q_2^2}{4}=\frac{64p_4^3p_0-16p_4^2p_2^2+16p_4p_3^2p_2-16p_4^2p_3p_1-3p_3^4}{64p_4^4}=\frac{\text{disc}_2(P)}{64p_4^4}.\]

Observe that the nature of the roots of $P$ and $Q$ are same. Let
\[
\ell_1=\{y=x^4+q_2x^2+q_0\}\quad\text{and}\quad\ell_2=\{y=-q_1x\}.
\] 
Since $\text{disc}(P)>0$, it follows that $Q$ has either four real roots or two pairs of complex conjugate roots. 
If the minimum value of curve $\ell_1$ is negative then the line $\ell_2$ will intersect $\ell_1$, 
and hence $Q$ will have four real roots. 
If $q_2<0$ then the minimum value of $\ell_1$ is $q_0-q_2^2/4$, and hence $q_2<0$ and $q_0-q_2^2/4<0$ imply (2). 
If $q_2>0$ or $q_0-q_2^2/4>0$ then $\ell_1$ and $\ell_2$ intersect at most twice,
and hence $Q$ has at most two real roots. 
But $Q$ has either four real roots or two pairs of complex conjugate roots. This proves (3).
\end{proof}

\end{appendix}

\bibliographystyle{amsalpha}
\bibliography{stability.bib}

\end{document}